\documentclass[12pt]{amsart}
\usepackage[utf8]{inputenc}
\usepackage{amsfonts}
\usepackage{latexsym}
\usepackage{amssymb}
\usepackage{amsmath}
\usepackage{color}

\usepackage{tikz}
\usepackage{enumerate} 
\usepackage{mathrsfs}
\usepackage{todonotes}

\usepackage[foot]{amsaddr}

\usepackage{etoolbox}
\usepackage{hyperref}

\usepackage{forest}
\forestset{
  default preamble={
   for tree={circle, draw, 
            minimum size=2.1em, % <-- added
            inner sep=1pt} 
  }
}

\definecolor{col2}{rgb}{0.95, 0.95, 0.65}

\definecolor{col1}{rgb}{0.3.0.4,0.6}

\definecolor{tc}{rgb}{0.00,0.36,0.26}
\definecolor{tc2}{rgb}{0.77,0.05,0.05}

\usepackage[bottom]{footmisc}

\usepackage[left=1.4cm, top=2.5cm,bottom=2.5cm,right=1.4cm]{geometry}

\newcommand{\R}{\mathbb R}
\newcommand{\N}{\mathbb N}
\newcommand{\T}{\mathbb T}

\newcommand{\E}{\mathbb E}

\renewcommand{\P}{\mathbb P}

\newcommand{\Var}{\mathrm{Var}}

\newtheorem{thm}{Theorem}[section]

\newtheorem{cor}[thm]{Corollary}
\newtheorem{lemma}[thm]{Lemma}
\newtheorem{df}[thm]{Definition}

\newtheorem{thmalpha}{Theorem}

{
\theoremstyle{definition}

}
\definecolor{gurot}{RGB}{180,20,20}
\definecolor{jorot}{RGB}{220,20,20}
\definecolor{darkorange}{rgb}{1.0, 0.55, 0.0}
\definecolor{folly}{rgb}{1.0, 0.0, 0.31}

\usepackage{nomencl}
\makenomenclature
\begin{document}

%%%%%%%%%%%%%%%%%%%%%%%%%%%%%%%%%%%%%%%%%%%%%5

{
\title[]{Sharp concentration for the largest and smallest fragment \\in a $k$-regular self-similar fragmentation}
}

\author{Piotr Dyszewski, Nina Gantert, Samuel G.~G.~Johnston, Joscha Prochno, Dominik Schmid}
\date{\today}

\begin{abstract}
We study the asymptotics of the $k$-regular self-similar fragmentation process. For $\alpha > 0$  and an integer $k \geq 2$, this is the Markov process $(I_t)_{t \geq 0}$ in which each $I_t$ is a union of open subsets of $[0,1)$, and independently each subinterval of $I_t$ of size $u$ breaks into $k$ equally sized pieces at rate $u^\alpha$. Let $k^{ - m_t}$ and $k^{ - M_t}$ be the respective sizes of the largest and smallest fragments in $I_t$. 
%% It transpires that this process can be well described using $q$-combinatorics, and by relating the fragmentation process to a time-inhomogeneous branching random walk (where the vertices in the genealogical tree in generation $n$ correspond to the intervals in the fragmentation of size $k^{-n}$) we are able to garner a precise understanding of the sizes of the smallest and largest fragments in the process at large times. 
By relating $(I_t)_{t \geq 0}$  to a branching random walk, we find
%in particular 
that 
%at large times the random functions $m_t$ and $M_t$ exhibit extremely concentrated behaviour, in that 
there exist explicit deterministic functions $g(t)$ and $h(t)$ such that $|m_t - g(t)| \leq 1$ and $|M_t - h(t)| \leq 1$ for all sufficiently large $t$. 
% - for most $t$ -  on a specific integer and - for some $t$ - on two neighbouring integers.[[I know this is less descriptive, but its more easily read.]] 
Furthermore, for each $n$, we study the final time at which fragments of size $k^{-n}$ exist. In particular, by relating our branching random walk to a certain point process, we show that, after suitable rescaling, the laws of these times converge to a Gumbel distribution as $n \to \infty$.  
\end{abstract}

\maketitle

%%%%%%%%%%%%%%%%%%%%%%%%%%%%%%
%%%%%%%%%%%%%%%%%%%%%%%%%%%%%%
%%%%%%%%%%%%%%%%%%%%%%%%%%%%%%
%%%%%%%%%%%%%%%%%%%%%%%%%%%%%%
\section{Introduction} \label{sec:intro}
%%%%%%%%%%%%%%%%%%%%%%%%%%%%%%
%%%%%%%%%%%%%%%%%%%%%%%%%%%%%%
%%%%%%%%%%%%%%%%%%%%%%%%%%%%%%
%%%%%%%%%%%%%%%%%%%%%%%%%%%%%%

%%%%%%%%%%%%%%%% % % % % % % % % % % % %
%\subsection{Fragmentation processes}
%%%%%%%%%%%%%%%% % % % % % % % % % % % %

Eighty years ago, Kolmogorov \cite{K1941} initiated the study of fragmentation processes, stochastic processes modelling an object of unit mass that breaks apart as time passes. 
%It was Kolmogorov's student Filippov who went on to study self-similar fragmentations \cite{filippov}, in which the rate of fragmentation of a particle may depend on its size. 
While research in fragmentation processes continued into the latter half of the 20th century  \cite{athreya1985discounted,brennan1986splitting,brennan1987splitting,filippov}, it was not until pathbreaking work by Bertoin \cite{B2001PTRF,bertoin} and Berestycki \cite{berestycki2002ranked} in the early 2000s that fragmentation processes were conceived in a unifying framework. This framework formulates a fragmentation process in terms of a stochastic process $(Y_t)_{t \geq 0}$ taking values in the set
\begin{align} 
\mathcal{S} := \left \{ (s_1,s_2,s_3,\ldots ) \,:\, s_1 \geq s_2 \geq \ldots \geq 0 , \sum_{ i = 1}^\infty s_i \leq 1 \right\},
\end{align}
whose law is governed by a \emph{dislocation measure} $\nu$ on the set $\mathcal{S}$. Writing $Y_t = (y_1(t),y_2(t),\ldots)$, the components $y_1(t) \geq y_2(t) \geq \ldots$ of $Y_t$
 correspond to the sizes of the fragments in the process at time $t$ listed in decreasing order.

In the setting where $\nu$ is finite, the homogeneous fragmentation processes first introduced by Bertoin \cite{B2001PTRF} have a simple description in terms of the dislocation measure:
each fragment of size $u$ has an exponentially distributed lifetime with rate $\nu(\mathcal{S})$, and upon death is replaced by a random collection of fragments of sizes $us_1 \geq us_2 \geq \ldots $, where the sequence $(s_1,s_2,\ldots)$ is distributed according to $\nu( \cdot) / \nu(\mathcal{S})$. In this context, \emph{homogeneous} refers to the fact that the rate at which each fragment breaks is independent of its size, and that the lifetimes and dislocations of individual fragments are independent of the remainder of the system. We remark that in general the measure $\nu$ need not be finite; indeed, infinite dislocation measures may be used to describe the continuous `crumbling' of fragments \cite{B2001PTRF}.

In the following, we will be interested in \emph{self-similar fragmentation processes}, in which fragments behave independently but the rate at which a fragment of size $u$ breaks apart is proportional to $u^\alpha$ for some $\alpha$ in $\mathbb{R}$. 
Self-similar fragmentations were introduced by %Kolmogorov's student  
Filippov \cite{filippov}, with their rigorous formulation in terms of general dislocation measures first appearing in \cite{bertoin}. 
The real parameter $\alpha$ is a called \emph{the index of self-similarity}, with $\alpha > 0$ entailing that larger fragments in the process break more quickly than smaller ones, and $\alpha < 0$ entailing the opposite. %Shortly, we will restrict ourselves to self-similar fragmentation processes with a positive index of self-similarity and a simple fragmentation mechanism. Before doing this however, we would like to foreground our results by overviewing related work on self-similar processes.

Brennan and Durrett \cite{brennan1986splitting,brennan1987splitting} study the self-similar fragmentation process where, upon death, a fragment of mass of size $u$ splits into exactly two fragments of sizes $Vu$ and $(1-V)u$, where $V$ is uniformly distributed on $[0,1]$.  They show that at large times $t$ the total number of intervals in the process grows in the order $t^{1/\alpha}$ for $0<\alpha<\infty$. 
%Following the formulation in Bertoin \cite{bertoin}, several authors have since studied self-similar fragmentations with general dislocation measures.}
Goldschmidt and Haas \cite{goldschmidt2010behavior, goldschmidt2016behavior} look at the explosive case $\alpha < 0$, in which after a finite amount of time the entire process consists of \emph{dust} so that there are no intervals of positive size. A work of particular relevance is the article \cite{bertoinJEMS} of Bertoin, where it is shown that if $y_1(t)$ is the size of the largest fragment in a self-similar fragmentation with $\alpha > 0$, then
\begin{align} \label{eq:Bertoin}
\lim_{t \to \infty} \frac{ \log y_1(t) }{ \log t } = - \frac{1}{\alpha} \qquad \text{almost surely}.
\end{align} 
See also the recent work of Dadoun \cite{D2017} for growth-fragmentation processes. While a panoply of exotic dislocation mechanisms fall into the general apparatus of self-similar fragmentation processes, in the present article we will concentrate our attention on the simplest possible fragmentation mechanism:
\begin{df}\label{def:kfrag}
	Fix an integer $k\geq 2$. The $k$-regular self-similar fragmentation process of index $\alpha\in\R$ is the self-similar fragmentation process  $(I_t)_{t \geq 0}$ starting with the single interval $I_0 := [0,1)$ in which an interval of size $u\in(0,1]$ in $I_t$ waits an exponential time with mean $u^{-\alpha}$, and after this time breaks into $k$ equally sized intervals.
\end{df}

Note that by listing the sizes of the intervals of $(I_t)_{t \geq 0}$ in decreasing order, $(I_t)_{t \geq 0}$ gives rise to an $\mathcal{S}$-valued process $(Y_t)_{t \geq 0}$.
The dislocation measure associated with the $k$-regular case belongs to a form of dislocation measures which Goldschmidt and Haas \cite{goldschmidt2016behavior} call `geometric', in that fragment sizes always take the form of a geometric progression $( r^n : n \in \mathbb{N})$ for some $r \in (0,1)$. Goldschmidt and Haas remark that geometric fragmentation processes possess genuinely different properties from non-geometric fragmentations, and should not be regarded as a degenerate special case. The reader is referred to \cite[Section 8]{goldschmidt2016behavior} for a discussion, wherein various other relevant references may be found, e.g.\ Athreya \cite{athreya1985discounted}.

The  relative simplicity of the mechanism means that we are endowed with a variety of exact formulas associated with various functionals of the processes, most notably allowing us to study an alternative representation for the process, where the fragments of sizes $k^{-n}$ are viewed as the $n^{\text{th}}$ generation of a discrete 
$k$-ary tree. 
%time branching process. 
These exact formulas lead to sharp statements about the asymptotics of the size of the smallest and largest fragments in the process at large times.
%and therefore to a considerably deeper understanding on the level of those processes.

In the remainder of the paper, we restrict our attention to the case $\alpha > 0$. Before stating our results in full in Section \ref{sec:results}, we conclude the introduction by giving the principal applications of our main results, showing that we can characterise the sizes of both the largest and smallest fragment at large times to a surprising degree of precision.
{In the sequel we write $\lceil x \rceil$ for the least integer greater than a real number $x$, and will denote by $\R_+$ the set of non-negative real numbers $[0,\infty)$. Finally, let us introduce the parameters
\begin{align*}
\gamma := \log k \qquad \text{and} \qquad \kappa := \frac{1}{\gamma \alpha}.
\end{align*}
Our main result on the largest fragment is a considerable sharpening of Bertoin's estimate \eqref{eq:Bertoin}, stating that if $k^{ - m_t}$ is the size of the largest fragment at time $t$, then $m_t$ has very concentrated behaviour.}

{
\begin{thmalpha} \label{thm:largest}
Let $k^{ - m_t}$ be the size of the largest fragment in the system at time $t$. Then for most times $t$, $m_t$ is likely to be the smallest integer above
\begin{equation*}
g(t) = \kappa \left( \log t - \log \log t -\log ( \gamma \kappa) \right).
\end{equation*}
More precisely, let $\mu_1 := \kappa + \frac{2}{\gamma}$. Then there exists almost surely a $t_0 \in \mathbb{R}_+$ such that for all $t \geq t_0$
\begin{equation*}
m_t \in \left\{ \left\lceil g(t) - \mu_1 \frac{ \log \log t }{ \log t} \right\rceil ,   \left\lceil g(t) + \mu_1 \frac{\log \log t}{ \log t}  \right\rceil \right\}.
\end{equation*}
\end{thmalpha}
Roughly speaking, for most values of $t$ the quantities $\left\lceil g(t) - \mu_1 \frac{ \log \log t }{ \log t} \right\rceil$ and $\left\lceil g(t) + \mu_1 \frac{ \log \log t }{ \log t}  \right\rceil$ concide, so that Theorem \ref{thm:largest} guarantees that for such $t$ we have $m_t = \lceil g(t) \rceil$. Occasionally an integer $n$ separates $g(t) - \mu_1 \frac{ \log \log t }{ \log t }$ and $g(t) + \mu_1 \frac{ \log \log t }{ \log t }$; it is in these time windows that $m_t$ has an opportunity to `jump' from $n$ to $n+1$. 

%In fact, we will see in Section \ref{sec:largest statements} that Theorem \ref{thm:largest} is merely one aspect of a far more descriptive characterisation of the largest fragments in the process in terms of a certain point process.
% related to the aforementioned discrete time branching process. \\

We now turn our attention to the size $k^{ - M_t}$ of the smallest fragment. Here we find that for large $t$, the law of the random variable $M_t$ is also highly concentrated:

\begin{thmalpha} \label{thm:smallest}
Let $k^{ - M_t}$ be the size of the smallest fragment in the system at time $t$. Then for most times $t$, $M_t$ is likely to be the smallest integer above
\begin{equation*}
h(t) :=  \kappa \left( \log t + \sqrt{2\gamma\log t}-\frac{1}{2}\log\log t + c  \right),
\end{equation*}
where $c := -\frac{1}{2\kappa}-\log \kappa +\gamma-\frac{1}{2}\log(2\gamma)+1$ is a constant. More precisely, let $\mu_2 := 2 \kappa^{2/3}$. Then there exists almost surely a $t_0 \in \mathbb{R}_+$ such that for all $t \geq t_0$
\begin{equation*}
M_t \in \left \{  \left \lceil h(t) - \mu_2 \frac{1}{ \log^{1/3} t} \right\rceil , \left \lceil  h(t) + \mu_2 \frac{1}{ \log^{1/3} t} \right\rceil \right\}.
\end{equation*}
\end{thmalpha}
}
%\jo{Again, I think we should say something about the idea of proof here!}
%\dom{For the proof of Theorem \ref{thm:smallest}, our strategy is to computed the expected number of fragments of a certain size and apply the second moment method.}

\vskip 2mm

The rest of the paper is organised as follows. In Section~\ref{sec:results}, we describe the representation of the fragmentation process 
as a certain time-inhomogeneous branching random walk, which is key to our proofs. We also give our main results on point process convergence (Theorem \ref{thm:largestpp} and Corollary \ref{cor:largestgumbel}) for the branching random walk, which explain why the sizes of the largest and smallest fragment satisfy such a sharp concentration property.
In Section~\ref{sec:typical}, we study weighted sums of exponential random variables and their relation to the $q$-Markov chain: the increasing Markov chain on $\{0,1,2,\ldots\}$ which jumps from a site $j$ to $j+1$ at rate $q^j$ (in our case, $q= k^{-\alpha}$). {In Section~\ref{sec:largest}, we prove our results on the largest fragments in the process. The final two sections, Section \ref{sec:smalls} and Section \ref{sec:smallest}, are dedicated to our work on the smallest fragments in the process.}

%%%%%%%%%%%%%%%%%%%%%%%%%%%%%%
%%%%%%%%%%%%%%%%%%%%%%%%%%%%%%
%%%%%%%%%%%%%%%%%%%%%%%%%%%%%%
%%%%%%%%%%%%%%%%%%%%%%%%%%%%%%
\section{The {associated} branching random walk} \label{sec:results}
%%%%%%%%%%%%%%%%%%%%%%%%%%%%%%
%%%%%%%%%%%%%%%%%%%%%%%%%%%%%%
%%%%%%%%%%%%%%%%%%%%%%%%%%%%%%
%%%%%%%%%%%%%%%%%%%%%%%%%%%%%%

This section is dedicated to giving a complete statement of our main results in their general form. 
%Before providing a more complete discussion of the asymptotic behavior of the largest and smallest fragments as discussed in the introduction, we begin in the next section by providing an alternative way of looking at the fragmentation process, which will be essential later. 
Through the majority of the proofs in the paper, we consider the fragments as vertices in a $k$-ary tree, where the offspring of an interval are the $k$ intervals it splits into.
We study the time of the appearance of the fragments using the fact that the fragmentation can be represented as a certain inhomogeneous branching random walk which we shall now describe.

%%%%%%%%%%%%%%%%
\subsection{Representation as BRW} \label{sec:time}
%%%%%%%%%%%%%%%%

We now explain the representation of the process in terms of an expanding branching random walk, see ~\cite{athreya1985discounted}. For the $k$-regular self-similar fragmentation of index $\alpha$, we set
\begin{equation}\label{qdef}
q: = k^{-\alpha},
\end{equation}
and note that $q< 1$ under our assumption $\alpha \in(0,\infty)$.
Let $k^{-X_t}$ be the size of the fragment containing $0$, in other words the fragment of the form 
$\left[0, k^{-X_t} \right)$, present in the system at time $t$. Then the process $(X_t)_{t \geq 0}$ forms a~Markov chain on $\{0,1,2, \ldots\}$ satisfying $X_0=0$ and
%At first we look at the simple Markov chains where particles jump to the right at site-dependent rates. Namely, let $\left( \lambda_i : i =0,1,2,\ldots \right)$ be a collection of positive reals, and consider the Markov chain  satisfying
\begin{align*}
\lim_{ h \downarrow 0} \frac{1}{h}\, \mathbb{P} \left( X_{t+h} = j | X_t = i \right) =
\begin{cases}
\lambda_i \qquad &\text{if $j = i+1$},\\
0 \qquad &\text{otherwise},
\end{cases}
\end{align*}
where $\lambda_i=q^i$.   
For $n\in\N$, define the time of fragmentation of the fragment $\left[0, k^{-n} \right)$ into $k$ fragments of sizes $k^{-(n+1)}$ to be $S_n := \sup \{ t \geq 0 : X_t = n \}$. It follows that $\{ X_t = n \} = \{ S_n > t, S_{n-1} \leq t \}$, and that $\{ X_t \leq n \} = \{ S_n > t \}$. 
Moreover, we may write 
\begin{equation}\label{eq:2:Sn}
	S_n = \sum_{ i = 0}^n \lambda_i^{-1} W_i,
\end{equation}	 
where, for each $i\in\{0,1,\dots,n\}$, $\lambda_i^{-1}W_i$ is the amount of time $X_t$ spends at the state $i$, and hence $W_i$ is a standard exponential random variable. 

The same analysis can be carried through on every interval. The dependence structure in the resulting system can be described using  branching processes. Each interval breaks into $k$ pieces and thus we may consider each interval $v$ of size 
$k^{ - n}$ living for some time period as a vertex $v$ within the $n^{\text{th}}$ generation of a $k$-regular tree. Indeed, let $\T_n$ denote the set of subintervals of the form $[m/k^n , (m+1)/k^n)$, $m\in\{0,1,\dots,k^n-1\}$ so that 
$\T_n$ has $k^n$ elements. Write $\T = \bigcup_{n \in \N}\T_n$ for the set of all subintervals that can appear in the system. For $v \in \T$, let $k^{-|v|}$ denote the size of $v$, in other words $|v|=n$ for $v \in \T_n$. Finally for intervals $v, w \in \T$ let $v \wedge w$ denote the smallest (in the sense of inclusion) element of $\T$ containing $v$ and $w$. Then $v \wedge w$ is the most recent common ancestor of $v$ and $w$. We will also write $v \geq w$ whenever $v \subseteq w$. 
Letting $S(v)$ denote the time at which an element $v$ of $\T_n$ of size $k^{ -n}$ breaks into $k$ pieces of sizes 
$k^{ - (n+1)}$, we now see that the set $\left\{ S(v) : v \in \T_n \right\}$ coincides with the positions of the $n^{\text{th}}$ generation of a certain branching random walk in which the step size distribution changes from generation to generation. 
The time of the first splitting is a standard exponential 
random variable $S([0,1)) = W^{([0,1))}$ and each particle in generation $n$ has exactly $k$ children in generation $n+1$. If $w \in \T_{n+1}$ is a child of $v\in\T_n$, then
\begin{align}\label{recforS}
S(w) = S(v) + q^{ - |w|} W^{(w)},
\end{align}
where $W^{(w)}$ is a standard exponential random variable which is independent of $S(v)$, recalling $q<1$. In fact, the random variable $q^{ - |w|} W^{(w)}$ is equal to the length of time that the interval $w$ exists in the process until it splits. Since $q < 1$, this random walk gets slower and slower as $n$ becomes large. We will refer to $S= (S(v))_{v \in \T}$  as the \textit{expanding branching random walk} as proposed in~\cite{athreya1985discounted}. It is natural to consider for all $v \in \T$ the rescaled quantities
\begin{align*}
 K(v)  := q^{|v|} S(v) 
\end{align*}
We will refer to $(K(v))_{v \in \T}$ as the \textit{rescaled expanding branching random walk}. The jumps in the rescaled branching random walk have the simple description that if $w$ is a child of $v$, then
\begin{align*}
K(w) = qK(v) + W^{(w)};
\end{align*}
that is, a particle inherits $q$ times their parent's position, plus a standard exponential. It is easily seen that as $n$ becomes large, for a typical $v \in \T_n$, the sum $K(v)$ has order $1$. In fact, the marginal law of each random variable $K(v)$ for $v \in \T_n$ is equal in distribution to the weighted sum
\begin{equation}\label{eq:2:kndef}
K_n := \sum_{ i = 0}^{n} q^i W_i,
\end{equation} 
where the $W_i$ are i.i.d.~standard exponential random variables. The collection $\{K_n\}_{n\in\N}$ forms a perpetuity sequence with almost sure limit $K_\infty$ being the solution to
\begin{equation}\label{eq:2:rde}
	K_\infty \stackrel{d}{=} q K_\infty + W, \qquad \mbox{$K_\infty$ independent from $W$,}
\end{equation}
%We will show in Section \ref{sec:typical} that the distribution of $K_n$ may be given in terms of the combinatorial formula
%\begin{align} \label{eq:K law}
%\mathbb{P} \left( K_n > t \right) = \sum_{ j = 0}^{n} \frac{ (-1)^j q^{j(j+1)/2} }{ \varphi_j(q) \varphi_n(q) } \exp \left( - q^{ - j } t \right).
%\end{align}
%It is straightforward to see that as $n \to \infty$, $K_n$ converges in distribution to a random variable $K_\infty$ with expectation $1/(1-q)$, and upper tails given by the formula
%\begin{align} \label{eq:Kinfty law}
%\mathbb{P} \left( K_\infty > t \right) = \frac{1}{ \varphi_\infty(q) } \sum_{ j = 0}^{\infty} \frac{ (-1)^j q^{j(j+1)/2} }{ \varphi_j(q) } \exp \left( - q^{ - j } t \right).
%\end{align}
where $W$ is a standard exponential random variable. Random variables of this type were studied in the literature~\cite{denisov2007theorem,rootzen1986extreme} with a heavy emphasis on the right tail behaviour $\mathbb{P}(K_\infty>t)$ as $t \to \infty$. A careful and delicate analysis of the upper and lower tails of $K_\infty$ will play an important role in our study of the asymptotics of the largest and smallest fragments of the process.

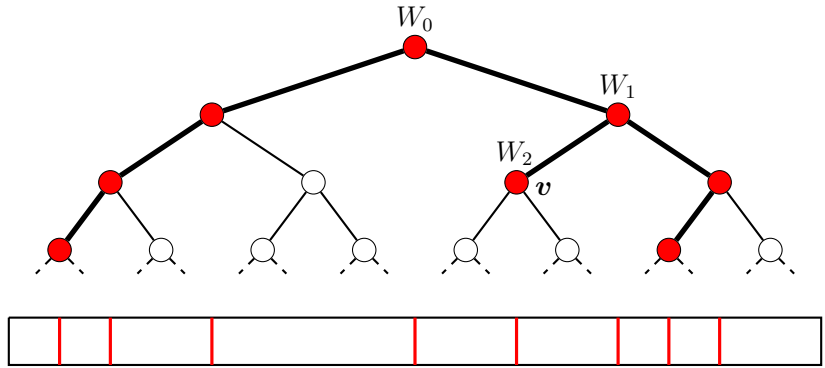
\begin{figure}

\begin{center}
\begin{tikzpicture}[scale=0.9]

	\node[shape=circle,scale=0.8,draw,fill=red] (A1) at (0,0){} ;
 	\node[shape=circle,scale=0.8,draw,fill=red] (A2) at (-3,-1){} ;
	\node[shape=circle,scale=0.8,draw,fill=red] (B2) at (3,-1) {};
	
	\node[shape=circle,scale=0.8,draw,fill=red] (A3) at (-4.5,-2){} ;
	\node[shape=circle,scale=0.8,draw] (B3) at (-1.5,-2){} ;
	\node[shape=circle,scale=0.8,draw,fill=red] (C3) at (1.5,-2){} ;
	\node[shape=circle,scale=0.8,draw,fill=red] (D3) at (4.5,-2){} ;
	
	\node[shape=circle,scale=0.8,draw,fill=red] (A4) at (-5.25,-3){} ;
	\node[shape=circle,scale=0.8,draw] (B4) at (-3.75,-3){} ;
	\node[shape=circle,scale=0.8,draw] (C4) at (-2.25,-3){} ;
	\node[shape=circle,scale=0.8,draw] (D4) at (-0.75,-3){} ;
	\node[shape=circle,scale=0.8,draw] (E4) at (0.75,-3){} ;
	\node[shape=circle,scale=0.8,draw] (F4) at (2.25,-3){} ;
	\node[shape=circle,scale=0.8,draw,fill=red] (G4) at (3.75,-3){} ;
	\node[shape=circle,scale=0.8,draw] (H4) at (5.25,-3){} ;

	\draw[line width=2pt] (A1) to (A2);
	\draw[line width=2pt] (A1) to (B2);
	\draw[line width=2pt] (A2) to (A3);
	\draw[thick] (A2) to (B3);
	\draw[line width=2pt] (B2) to (C3);
	\draw[line width=2pt] (B2) to (D3);	
	
	\draw[line width=2pt] (A3) to (A4);
	\draw[thick] (A3) to (B4);
	\draw[thick] (B3) to (C4);
	\draw[thick] (B3) to (D4);	
	\draw[thick] (C3) to (E4);
	\draw[thick] (C3) to (F4);
	\draw[line width=2pt] (D3) to (G4);
	\draw[thick] (D3) to (H4);			
	
	\draw[thick, dashed] (A4) to (-5.6,-3.35);		
	\draw[thick, dashed] (A4) to (-4.9,-3.35);
	\draw[thick, dashed] (B4) to (-4.1,-3.35);		
	\draw[thick, dashed] (B4) to (-3.4,-3.35);			
	\draw[thick, dashed] (C4) to (-2.6,-3.35);		
	\draw[thick, dashed] (C4) to (-1.9,-3.35);			
	\draw[thick, dashed] (D4) to (-1.1,-3.35);		
	\draw[thick, dashed] (D4) to (-0.4,-3.35);	
		
	\draw[thick, dashed] (E4) to (0.4,-3.35);		
	\draw[thick, dashed] (E4) to (1.1,-3.35);
	\draw[thick, dashed] (F4) to (1.9,-3.35);		
	\draw[thick, dashed] (F4) to (2.6,-3.35);			
	\draw[thick, dashed] (G4) to (3.4,-3.35);		
	\draw[thick, dashed] (G4) to (4.1,-3.35);			
	\draw[thick, dashed] (H4) to (4.9,-3.35);		
	\draw[thick, dashed] (H4) to (5.6,-3.35);

	\draw[thick] (-6,-4) -- (6,-4) -- (6,-4.7) -- (-6,-4.7) -- (-6,-4);			

   \draw[line width=1.2pt,red] (-5.25,-4) -- (-5.25,-4.7);
   \draw[line width=1.2pt,red] (-4.5,-4) -- (-4.5,-4.7);
   \draw[line width=1.2pt,red] (-3,-4) -- (-3,-4.7);
   \draw[line width=1.2pt,red] (0,-4) -- (0,-4.7);
   \draw[line width=1.2pt,red] (1.5,-4) -- (1.5,-4.7);
   \draw[line width=1.2pt,red] (3,-4) -- (3,-4.7);
   \draw[line width=1.2pt,red] (3.75,-4) -- (3.75,-4.7);
   \draw[line width=1.2pt,red] (4.5,-4) -- (4.5,-4.7);

		\node[thick, scale=0.9]  at (1.9,-2.1){$\boldsymbol{v}$};	
		
		\node[scale=0.9]  at (1.47,-1.57){$W_2$};	
		\node[scale=0.9]  at (3,-0.58){$W_1$};		
		\node[scale=0.9]  at (0,0.42){$W_0$};

\end{tikzpicture}
\end{center}
\caption{\label{fig:Branching}  Visualization of a $2$-regular self-similar fragmentation of index $\alpha$, and the genealogical tree of its associated branching random walk at some time $t\geq0$. All sites present in the tree at time $t$ are marked in red. Note that the i.i.d.~standard exponential random variables $(W_i)_{i\in\N}$ in the definition of  $S(v)$ for the site $v$ must satisfy $\sum_{i=0}^{|v|} 2^{\alpha i }W_i \leq t$.}
\end{figure}

We conclude this section on the representation with a branching random walk emphasizing the scaling on which the process may be viewed. Indeed, consider the interval $[0,k^{-n})$ -- a representative of the typical interval of size $k^{ - n}$ -- which exists for a random period of time during the process. This random period of time is equal in law to 
\begin{align*}
[ q^{ - (n-1)} K_{n-1} , q^{ - (n-1)} K_{n-1} + q^{ -n} W ),
\end{align*}
where $W$ is a standard exponential random variable and $K_n$ is given by~\eqref{eq:2:kndef} (so that in particular, $K_n$ has unit order when $n$ is large). In particular, loosely speaking we have:
\begin{align*}
\text{The times $t $ for which the intervals of size $k^{-n}$ exist in the process are of order $q^{ - (1 + o(1)) n}$ }.
\end{align*} 
Inverting this relation gives:
\begin{align*}
\text{The intervals of sizes $k^{ - n}$ existing at a time $t$ have the order $n = (1 + o(1)) \kappa \log t$}.
\end{align*}
where, as in the introduction, $\kappa = 1 / \log (1/q)$. {In particular, this discussion sketches the first order scale on which the process lives: the typical interval at time $t$ has size $k^{ - (1 + o(1)) \kappa \log t}$. Note that Theorem \ref{thm:largest} and Theorem \ref{thm:smallest}  state that indeed  \emph{every} interval has this size. }

%With this picture in mind, we are equipped to say \emph{the typical interval at time $t$ has size $k^{ - (1 + o(1)) \kappa \log t }$}. Indeed, as mentioned in the introduction, the largest fragments in the process have sizes of the form $k^{ - \kappa \left( \log t - \log \log t + o(\log \log t ) \right)}$, where as the smallest fragments have sizes of the form $k^{  - \kappa \left( \log t + \sqrt{ 2 \gamma \log t } + o( \sqrt{ \log t } \right) }$. 

%%%%%%%%%%%%%%%%
\subsection{Largest fragments in the process} \label{sec:largest statements}
%%%%%%%%%%%%%%%%

Recall Theorem \ref{thm:largest} in the introduction, which stated that if $k^{ - m_t}$ is the size of the largest fragment in the process at time $t$, then with high probability for all large times $t$, $m_t$ is one of the integers neighbouring the quantity
\begin{align*}
\kappa \log t - \kappa \log \log t - \kappa \log ( \gamma \kappa ) .
\end{align*} 
In fact,  this is explained by a far more descriptive result, which we now elucidate from the branching random walk perspective.
Given an element $v$ of $\T_n$, for each $0 \leq i \leq n$, let $v_i$ be the unique ancestor of $v$ in generation $i$, i.e.\, in $\T_i$. One key property of the process $(K(v))_{v \in \T}$ is that the majority of mass in each quantity $K(v)$ is due to recent ancestors. Indeed, we have the representation
\begin{align}\label{repofK}
K(v) = \sum_{ i = 0}^{|v|} q^{|v|-i} W^{(v_i)},
\end{align}
so that most of the mass in $K(v)$ is due to recent ancestors of $v$: those terms $W^{(v_i)}$ where $i$ is close to $|v|$. Intuitively, this implies that to a large extent, the random variables $\left( K(v) : v \in \T_n \right)$ are asymptotically independent.
We note for further reference that \eqref{repofK} implies that for $m< n$, with $v_m$ denoting the ancestor of $v$ in generation $m$,
\begin{align}\label{recrepofK}
K(v) = q^{n-m} K(v_m) + \widetilde{K}_{n-m+1},
\end{align}
where $\widetilde{K}_{n-m+1}$ is independent of $K(v_m)$ and has the same law as $K_{n-m+1}$.
Moreover, it is not too hard to show (we do it in Section \ref{sec:typical}) that the upper tails of the $K_n$ take the form
\begin{align*} %\label{eq:uppertails}
\mathbb{P} \left( K_n > s \right) = (1 + o(1)) e^{ - s} / \varphi_n(q) \qquad \text{for large $s$},
\end{align*}
where
\begin{equation}\label{eq:varphi def}
	\varphi_n(q) := \prod_{j=1}^n(1-q^j),\,  n \geq 1, \qquad \varphi_0(q) :=1.
\end{equation}
In particular, the maximal elements of the collection $(K(v))_{v \in \T_n}$ behave a lot like the maximum of $k^n$ independent random variables with exponential tails: namely, like a Gumbel random variable. We mention in passing that $\varphi_n(q)$ is a decreasing function of $n$, 
and that as $n \to \infty$, $\varphi_n(q)$ converges to the Euler function $\varphi_\infty( q) := \prod_{ i = 1}^\infty ( 1 - q^i )$, which takes strictly positive values for $q \in (0,1)$. This may be seen, for instance, from Eulers pentagonal number theorem, see \cite{vardi1999introduction}, or from the well-known fact that for 
$0< a_i < 1$, we have $\Pi (1-a_i) > 0$ if and only if $\sum a_i < \infty$.\\

Let $N_n$ be the point process on the real line given by 
\begin{align}\label{def:NnAndJv}
N_n=\sum_{v \in \T_n}\delta_{J(v)}, \qquad J(v) := K(v) - \gamma |v|,
\end{align}
where we recall that $\gamma = \log k$. Our main result 
%on the largest fragments 
states that the elements of  $(J(v))_{v \in \T_n}$ behave like a Poisson point process on the real line. 

\begin{thmalpha} \label{thm:largestpp}
As $n \to \infty$, the point process $N_n$ converges in distribution (in the sense of vague convergence from \cite{last2017lectures}) to a Poisson point process with intensity measure
\begin{align} \label{eq:intensity}
e^{ - s} ds / \varphi_\infty(q).
\end{align}
Moreover, the neighbouring point processes are asymptotically independent, in the sense that for any $\ell \geq 1$, the vector of point processes $(N_n,\ldots,N_{n + \ell - 1})$ converges in distribution to a vector of $\ell$ independent Poisson processes with intensity given in \eqref{eq:intensity}.
\end{thmalpha}

Let us now consider the large fragments. 
One immediate consequence of Theorem \ref{thm:largestpp} is the following result on the asymptotic behaviour of
\begin{equation}\label{eq:2:deftau}
	K^{\mathrm{max}}_n := \max \left \{ K(v) : v \in \T_n \right\} \qquad \mbox{and} \qquad \tau_n := K^{ \mathrm{max}}_n - \gamma n= \max \left\{ J(v)  : v \in \T_n   \right\}.
\end{equation}

\begin{cor} \label{cor:largestgumbel}
 Let $\tau_n$ be defined as in~\eqref{eq:2:deftau}. Then, as $n \to \infty$, $\tau_n$ converges in distribution to a shifted Gumbel random variable, i.e.\,
\begin{align*}
\lim_{n \to \infty} \mathbb{P} \left( \tau_n \leq s \right) = \exp \left( - e^{ -s} / \varphi_\infty(q) \right).
\end{align*}
\end{cor} 
Corollary ~\ref{cor:largestgumbel} explains the concentration of the size of the largest fragment given in Theorem \ref{thm:largest}.
Note that we may write 
\begin{align} \label{eq:kol}
	\{ m_t \leq n \}= \{ q^{ - n } K_n^{ \mathrm{max}} > t \} .
\end{align}
 In particular, using the definition of $\tau_n$ in \eqref{eq:2:deftau}, we have 
\begin{align} \label{eq:newrep}
\mathbb{P} \left( m_t \leq n \right) =  \mathbb{P} \left( \tau_n >  q^n t- \gamma n  \right).
\end{align}
Now, since $\tau_n$ converges in distribution, if $q^n t- \gamma n \to \infty$ for $n \to \infty$, the probability on the right-hand side of \eqref{eq:newrep} goes to $0$, and if $q^n t- \gamma n \to -\infty$ for $n \to \infty$, the probability on the right-hand side of \eqref{eq:newrep} goes to $1$. In order to give a proof of the almost-sure statement Theorem \ref{thm:largest} we will need some uniform estimates for 
$\tau_n$ which we will develop in the sequel. The full proof of Theorem  \ref{thm:largest} is given at the  {beginning} of Section \ref{sec:largest}.
%We now explain how Theorem \ref{thm:largest} in the introduction may be obtained as a consequence of Corollary~\ref{cor:largestgumbel}. Note that we may write 
%\begin{align*}
%	\{ m_t \leq n \}= \{ S_n^{ \mathrm{max}} > t \},
%\end{align*}
%where $S_n^{ \mathrm{max}} := q^{ - n } K_n^{ \mathrm{max}}$. In particular, using the definition of $\tau_n$ in \eqref{eq:2:deftau} we have 
%\begin{align} \label{eq:newrep}
%	\mathbb{P} \left( m_t \leq n \right) =  \mathbb{P} \left( \tau_n >  q^n t- \gamma n  \right).
%\end{align}
%Now suppose $n=n_\delta(t) = \kappa \left( \log t - \log \log t + \delta \right)$ is an integer. Then we have
%\begin{align*}
%	q^n t-\gamma n  =  e^{ - \delta} \log t -\gamma \kappa \left( \log t - \log \log t + \delta \right).
%\end{align*} 
%In particular, for constants $\delta$, $q^n t - \gamma n \to + \infty$ whenever $\delta \leq - \log ( \gamma \kappa)$ and $q^n t - \gamma n \to- \infty$ whenever $\delta > -\log ( \gamma \kappa)$. 
%In particular, by \eqref{eq:newrep}, since $\tau_n$ converges in distribution to a shifted Gumbel random variable, 
%\begin{align*}
%\mathbb{P} \left( m_t = g(t) \right) \to 1, \text{as $t \to \infty$},
%\end{align*}
%where $g(t)$ is given via~\eqref{eq:1:gt}.
%In order to give a proof of the slightly stronger almost-sure statement Theorem \ref{thm:largest} we will need some delicate estimates for $\tau_n$ which we will develop in the sequel.

%%%%%%%%%%%%%%%%
\subsection{Smallest fragments in the process} \label{sec:smallest statements}
%%%%%%%%%%%%%%%%

{
We saw in Section \ref{sec:largest statements} that the behaviour of the largest fragments in the $k$-regular self-similar fragmentation process is intimately connected with the largest values $K_n^{\max} := \max_{v \in \mathbb{T}_n } K(v)$ in the rescaling of the expanding branching random walk. Analogously, it is the behaviour of the smallest value $K_n^{\min} := \min_{ v \in \mathbb{T}_n } K(v)$ that ultimately dictates the asymptotics of the smallest fragments in the fragmentation process. In this direction we have the following result.

\begin{thm} \label{thm:smallestnew}
	Define $K^{\min}_n := \min \{ K(v) : v \in \T_n \}$ and define $w_n = w_n(\kappa,\gamma)$ by 
\begin{equation}\label{eq:znasymptotics}
		w_n := \sqrt{\frac{2\gamma}{\kappa}n} - \frac{1}{2}\log n - \frac{1}{2\kappa}- \frac{1}{2}\log\kappa +1 - \frac{1}{2}\log(2\gamma).
	\end{equation} 
Then there exists almost-surely an $n_0$ in $\mathbb{N}$ such that for all $n \geq n_0$ we have 
\begin{align*}
\log K_n^{\min} \in \left[ - w_n - \frac{1}{n^{1/3} },  - w_n + \frac{1}{ n^{1/3}}  \right].
\end{align*}
\end{thm}

In Section \ref{sec:smallest} we prove Theorem 2.2, and thereafter use Theorem 2.2 to prove Theorem \ref{thm:smallest}.

One of the key tasks in proving Theorem 2.2 is a careful analysis of the $s \downarrow 0$ asymptotics of the left tails $\mathbb{P}(K_\infty \leq s)$ of the random variable
\begin{equation}\label{eq:3:kndef}
K_\infty := \sum_{ i = 0}^{\infty} q^i W_i,
\end{equation} 
where $W_i$ are independent standard exponentials. Indeed, we note that since $K_n$ is stochastically dominated by $K_\infty$ (defined in \eqref{eq:2:kndef}), and $K_n$ is a sum of $n+1$ independent exponentials,  for any $n$ we have
\begin{align*}
\mathbb{P}( K_\infty \leq s ) \leq \mathbb{P}( K_n \leq s) \leq C(q,n) s^{n+1}
\end{align*}
for some $C(q,n)$ independent of $s$. In particular, as $s \downarrow 0$, the probability $\mathbb{P}(K_\infty < s)$ goes to zero faster than any power of $s$. The following result, which we believe to be of independent interest, gives a fine characterisation of these fast asymptotics. 

\begin{thm} \label{thm:smalls}
There exists a constant $C_q$ such that for all $s \in (0,1/e^2]$ and for all $n \geq \kappa \left( \log\frac{1}{s} + \log \log \frac{1}{s} \right)$, including possibly $n = \infty$, we have
\begin{align} \label{eq:smalls}
	\frac{1}{C_q} \exp \left( - F_q(s) \right)\leq \mathbb{P} \left( K_n \leq s \right) \leq C_q \exp \left( - F_q(s) \right),
\end{align}
where
\begin{align}\label{def:functionFqs}
F_q(s) := \frac{\kappa}{2} \left( \log \frac 1 s + \log \log \frac 1 s + \frac{1}{ 2 \kappa } + \log \kappa - 1 \right)^2 +\left( \frac{1}{2} + \kappa \right) \log \log \frac 1 s .
\end{align}
\end{thm}

Theorem \ref{thm:smalls} is proven in Section \ref{sec:smalls}. We remark that the restriction $s \leq 1/e^2$ ensures $\log \log \frac{1}{s} > 0$. Let us also note from Theorem \ref{thm:smalls} that for fixed $s$, provided $n$ is sufficiently large compared to $1/s$, the left tail $\mathbb{P}(K_n \leq s)$ takes the same order as $\mathbb{P}(K_\infty \leq s)$. 

That completes the section on statements of our main results. In the next section we begin setting the foundations for proofs of these statements by looking at formulas surrounding the random variables $K_n$ and the associated Markov chains. Thereafter we provide a simple lemma suitable for converting statements about the expanding branching random walk to those about the fragmentation process.}

\section{Preliminaries on the rescaled expanding branching random walk}  \label{sec:typical}
%%%%%%%%%%%%%%%%%%%%%%%%%%%%%%
%%%%%%%%%%%%%%%%%%%%%%%%%%%%%%
%%%%%%%%%%%%%%%%%%%%%%%%%%%%%%
%%%%%%%%%%%%%%%%%%%%%%%%%%%%%%

Throughout the rest of this paper, $C_\Omega\in(0,\infty)$ is a constant which is not of particular interest, and which may vary from line to line, but depends only on the set of parameters $\Omega \subseteq \{ q, \vec{p}, k, t_0 \}$ (with parameters $\vec{p}$ and $t_0$ yet to be defined). We stress that constants $C_\Omega$ do not depend on $n,m \in \N$ and $t>0$.  

%%%%%%%%%%%%%%%%%%
\subsection{Transition probabilities of birth processes} \label{sec:markov}
%%%%%%%%%%%%%%%%%%

Recall that $k^{-X_t}$ denotes the length of the interval containing $0$ present in the system at time $t$. As noted in Section~\ref{sec:time} the moment of the $n^{\rm th}$ splitting of this interval, $S_n = \sup \{ t \geq 0 \: : X_t=n \}$ has an explicit representation, see~\eqref{eq:2:Sn}.
Using \eqref{eq:2:Sn} one can compute directly (see for instance Feller \cite[I.13 Problem 12]{feller}) that
\begin{align} \label{eq:entrance}
\mathbb{P} \left( S_n \in dt \right) = \left( \prod_{i=0}^{n} \lambda_i \right)  \sum_{ j = 0}^{n} \frac{ e^{ - \lambda_j t } }{ \prod_{ 0 \leq k \leq n, k \neq j } (\lambda_k - \lambda_j ) }dt,
\end{align}
where $\lambda_i=q^i$.
Integrating both sides of \eqref{eq:entrance}, we obtain 
\begin{align} \label{eq:entrance2}
\mathbb{P} \left( S_n > t \right) = \left( \prod_{i=0}^{n} \lambda_i \right)  \sum_{ j = 0}^{n} \frac{ e^{ - \lambda_j t } }{ \lambda_j  \prod_{ 0 \leq k \leq n, k \neq j } (\lambda_k - \lambda_j ) }.
\end{align}
Consider now calculating $\mathbb{P} \left( X_t = n \right) = \mathbb{P} \left( S_n> t , S_{n-1} \leq t \right)$. We claim that 
\begin{align}  \label{eq:entrance3}
\mathbb{P} \left( X_t= n \right)  = \left( \prod_{i=0}^{n-1}  \lambda_i \right)  \sum_{ j = 0}^n \frac{ e^{ - \lambda_j t } }{ \prod_{ 0 \leq k \leq n, k \neq j } (\lambda_k - \lambda_j ) }. 
\end{align}
The most natural way to prove \eqref{eq:entrance3} is by writing $\mathbb{P} \left( X_t= n \right) = \mathbb{P} \left( S_n > t \right ) - \mathbb{P} \left( S_{n-1} > t \right)$, and then applying \eqref{eq:entrance2}. However, there is a far slicker route, writing 
\begin{align*}
	\mathbb{P}\left( S_n \in [t, t+h) \right) &= \mathbb{P}\left(  X_t\leq n, \: X_{t+h} > n \right) = \mathbb{P}\left(  X_t = n, \: X_{t+h} = n+1 \right)  + o(h) .
\end{align*}
Using the Markov property results in
\begin{align} \label{eq:jump}
\mathbb{P} \left( S_n \in dt \right) = \lambda_n \mathbb{P} \left( X_t = n \right)dt.
\end{align}
In particular, by using \eqref{eq:jump} and \eqref{eq:entrance}, we immediately obtain \eqref{eq:entrance3}. With a view towards tackling the  equations \eqref{eq:entrance},  \eqref{eq:entrance2} and  \eqref{eq:entrance3} with $\lambda_i = q^i$, 
recall the definition \eqref{eq:varphi def} of $\varphi_n(q)$, and calculate that for any $0 \leq j \leq n$,
\begin{align} \label{eq:q product identity}
\prod_{ 0 \leq k \leq n , k \neq j } \left( q^k - q^j \right) = (-1)^{n-j} q^{ j ( n - j/2 - 1/2)} \varphi_j(q) \varphi_{n-j} (q).
\end{align}
By replacing $j$ with $n-j$ and using the identity \eqref{eq:q product identity}, it is straightforward to show that, setting $\lambda_i = q^i$ in \eqref{eq:entrance2}, %and  \eqref{eq:entrance3} 
we have
\begin{align*} %\label{eq:q entrance} 
\mathbb{P} \left( S_n > t \right) = \sum_{  j = 0}^{n}  \frac{ (-1)^j q^{ j(j+1)/2 }}{ \varphi_j (q) \varphi_{n-j}(q) } \exp \left( - q^{n-j} t  \right).
\end{align*}
Recall that $K_n$ is given in \eqref{eq:2:kndef}, and is equal in distribution to $q^{n} S_n$. Thus, we have
\begin{align} \label{eq:K_N} 
\mathbb{P} \left( K_n > t \right) = \sum_{  j = 0}^{n}  \frac{ (-1)^j q^{ j(j+1)/2 }}{ \varphi_j (q) \varphi_{n-j}(q) } \exp \left( - q^{-j} t  \right).
\end{align}
By differentiating both sides of \eqref{eq:K_N}  with respect to $t$, we see that the density $f_n$ of $K_n$ is given by 
\begin{align}\label{def:Densityfn}
f_n(t) & =\sum_{  j = 0}^n   \frac{ (-1)^j q^{ j(j-1)/2 }}{ \varphi_j (q) \varphi_{n-j}(q) } \exp \left( - q^{-j} t  \right).
\end{align}
From \eqref{eq:2:kndef} it is plain that $K_n \leq K_{n+1}$, and that almost surely, as $n \to \infty$, the random variables $(K_n)_{n\in\N}$ converge to a finite limit $K_\infty$, which is given by
\begin{equation}\label{def:K_infty}
	K_\infty=\sum_{i=0}^\infty q^{i}W_i.
\end{equation}
It is straightforward to verify, using the monotone convergence theorem and \eqref{eq:K_N}, that 
\begin{align} \label{eq:Kinfinity} 
\mathbb{P} \left( K_\infty > t \right) = \frac{1}{ \varphi_\infty(q) } \sum_{  j = 0}^\infty   \frac{ (-1)^j q^{ j(j+1)/2 }}{ \varphi_j (q) } \exp \left( - q^{-j} t  \right).
\end{align}
That the right-hand side of \eqref{eq:Kinfinity} is equal to $1$ when $t = 0$ is a consequence of the identity
\begin{align*}% \label{eq:macdonald}
\sum_{ j = 0}^\infty \frac{ \zeta^j q^{j(j+1)/2} }{ \varphi_j(q) } = \prod_{ i = 1}^\infty ( 1 + q^i \zeta),
\end{align*}
which is a well known fact in $q$-combinatorics; see for instance Exercise 4 in Section I.2 of Macdonald \cite{macdonald1998symmetric}. 
%
%Later we will require the following bound. 
%\begin{lemma} \label{eq:productbound}
%Let $\xi q \leq 1/2$, then
%\begin{align*}
%\prod_{ i = 0}^\infty ( 1 + \xi q^{i} ) \leq 1 + C''_q \xi
%\end{align*}
%where $C''_q = 2 / \varphi_\infty(q)$. 
%\end{lemma}
%
%\begin{proof}
%By \eqref{eq:macdonald}
%\begin{align*}
%\prod_{ i = 0}^\infty ( 1 + \xi q^{i} ) = \sum_{ j = 0}^\infty \frac{ \xi^j q^{j(j+1)/2} }{ \varphi_j(q) }  \leq 1 + \frac{1}{ \varphi_\infty(q) } \sum_{ j =1 }^\infty ( \xi q)^j 
%\end{align*}
%which gives us the reuslt.
%\end{proof}
%
%In the remainder of Section \ref{sec:typical}, we now study the upper and lower tails.
%
% % % % % % % % % % % % % % % % % % % %
%\subsection{The upper tails of $K_n$}
% % % % % % % % % % % % % % % % % % % %
%In the next two lemmas, we will see that the upper tails $\mathbb{P} \left( K_n > t \right)$ of $K_n$ and its density behave roughly like that of a standard exponential random variable. 
%
Using~\eqref{eq:Kinfinity}, we can control the second order term in the asymptotic expansion of the right tail of $K_\infty$, which will be useful in the sequel.

\begin{lemma} \label{lem:tails}
		For every $t \geq 0$, we have the following tail and density bounds for $K_n$, $n \in \mathbb{N}\cup \{ \infty\}$,  
		\begin{align*} %\label{eq:tailbound}
			\left| \mathbb{P} \left( K_{n} > t \right) - \frac{ e^{ - t}}{ \varphi_{n}(q) } \right| \leq C_q e^{ - t/q}
		\end{align*}
		and 
		\begin{align*}% \label{eq:densitybound}
			\left|f_n(t) - \frac{ e^{ - t}}{ \varphi_n(q)} \right|\leq C_q  e^{ - t / q}.
		\end{align*}
\end{lemma}

\begin{proof}
Recall \eqref{eq:K_N}. We use the triangle inequality, and the facts that $\varphi_j(q)$ is decreasing in $j$, and that we have $q < 1$, to see that
\begin{align*}
\left| \mathbb{P} \left( K_n > t \right) - \frac{ e^{ - t}}{ \varphi_{n}(q) } \right|  &= \left|   \sum_{  j = 1}^{n}  \frac{ (-1)^j q^{ j(j+1)/2 }}{ \varphi_j (q) \varphi_{n-j}(q) } \exp \left( - q^{-j} t  \right) \right|\\
&\leq  \sum_{  j = 1}^{n} \frac{  q^{ j }   }{ \varphi_{n}(q)^2 } \exp \left( - t / q  \right)   .
\end{align*}
Now note that $\varphi_n(q) \geq \varphi_\infty(q)$, which gives the first claim. A similar argument yields the second {claim}. %one.
\end{proof}

We will also find occasion to use the crude bounds
\begin{align}\label{crude}
f_n(t) \leq C'_q e^{ -t} \qquad \text{and} \qquad \mathbb{P}( K_n < t) \leq C_q' e^{ - t} \qquad t\geq 0,
\end{align}
{both of which are direct consequences of Lemma~\ref{lem:tails}. Note that the latter bound can be significantly improved, as we will see  in Section \ref{sec:smalls} when proving Theorem \ref{thm:smalls}.}

{
\subsection{From the branching random walk back to the fragmentation process}
In this brief section we give a basic lemma for bounding values of increasing functions $f:\mathbb{R}_+ \to \mathbb{N}$ in terms of the times at which they jump. This allows us to convert the results on the rescaled expanding branching random walk  to statements about the fragmentation process, see Theorem \ref{thm:largest} and Theorem \ref{thm:smallest}. The proof follows from a standard computation and will therefore be omitted.

\begin{lemma} \label{lem:convert}
Let $t_0 \in \mathbb{R}_+$ and let $n_0 \in \mathbb{N}$. Suppose $f:[t_0,\infty) \to \{n_0,n_0+1,n_0+2,\ldots \}$ is a surjective and increasing right-continuous function. For each $n \geq n_0$ define
\begin{align*}
T_n := \sup \{ t \geq 0: f(t) = n \}
\end{align*}
to be the point in $[t_0,\infty)$ at which $f(t)$ jumps from $n$ to $n+1$.  Suppose there exist strictly increasing functions $a,b:[n_0,\infty) \to \mathbb{R}_+$ such that for each $n \geq n_0$ 
\begin{align*}
a(n) \leq T_n \leq b(n).
\end{align*}
Then for all $t \in [t_0,\infty)$ we have
\begin{align*}
f(t) \in \left\{ \lceil b^{ -1}(t) \rceil , \lceil a^{ -1}(t) \rceil \right\} 
\end{align*}
where $a^{-1}$ and $b^{-1}$ are the inverse functions of $a$,$b$ respectively. 
\end{lemma}
}

%%%%%%%%%%%%%%%%%%%%%%%%%%%%%%
%%%%%%%%%%%%%%%%%%%%%%%%%%%%%%
%%%%%%%%%%%%%%%%%%%%%%%%%%%%%%
%%%%%%%%%%%%%%%%%%%%%%%%%%%%%%
\section{The rightmost particles in the rescaled expanding branching random walk} \label{sec:largest}
%%%%%%%%%%%%%%%%%%%%%%%%%%%%%%
%%%%%%%%%%%%%%%%%%%%%%%%%%%%%%
%%%%%%%%%%%%%%%%%%%%%%%%%%%%%%
%%%%%%%%%%%%%%%%%%%%%%%%%%%%%%

{In this section we study the largest particles in the rescaled expanding branching random walk, which are connected to the largest fragments in the fragmentation process. We begin in Section \ref{sec:A proof} with a proof of Theorem A concerning the concentration in law of the size of the largest fragment at large times. In the remainder of Section \ref{sec:largest} we study point processes associated with the largest particles in the rescaled branching random walk, ultimately proving Theorem \ref{thm:largestpp}.}
\subsection{Proof of Theorem A} \label{sec:A proof}
{We now have all tools to prove Theorem A concerning the size $k^{ - m_t}$ of the largest fragment in the process. We recall from Section \ref{sec:results} that we associate with the fragmentation an expanding branching random walk $\{ S(v) : v \in \mathbb{T} \}$: the elements $v$ in the $n^{\text{th}}$ generation $\mathbb{T}_n$ of $\mathbb{T}$ correspond to the intervals of size $k^{ - n}$, with $S(v)$ denoting the time at which the interval $v$ fragments. In particular, the quantity 
\begin{align*}
\max_{v \in \mathbb{T}_n } S(v) = \sup \{ t \geq 0 : m_t \leq n \}
\end{align*}
is the last time at which there is an interval of size $k^{-n}$ in the process. We recall further that $K(v) := q^{|v|}S(v)$ denotes the rescaling of the expanding BRW, and that $\tau_n := \max_{v \in \mathbb{T}_n} K(v) - \gamma n$. In particular, up to scaling and translation, the behaviour of $\tau_n$ dictates that of the maximal fragment. 

We now obtain upper bounds on both the upper and lower tails of $\tau_n$. Considering first the upper tail, by using the union bound to obtain the first inequality below, and then the tail bound \eqref{crude} on $K_n$ to obtain the second, we have
	\begin{equation*}
		\P(\tau_n >s) = \P\left(\max_{|v|=n} J(v)>s \right) \leq k^n \P(K_n >s + \gamma n) \leq C_q e^{ - s},
	\end{equation*}
	for all $n \in \mathbb{N}$ and $s \in \mathbb{R}$. In particular, $\mathbb{P}( \tau_n > 2 \log n ) \leq C_q/n^2$ is summable in $n$, so that by the Borel-Cantelli lemma
	\begin{align}\label{tausmall}
		\P\left( \tau_n \leq 2 \log n \mbox{  for all but finitely many $n$} \right)=1.
	\end{align}
	On the other hand, by the construction of the rescaled expanding branching random walk (see \eqref{repofK}), 
$\max_{|v|=n} J(v)$ stochastically dominates $\max_{1\leq j \leq k^n} (W_j - \gamma n)$, where $W_1, \ldots , W_{k^n}$ are i.i.d.~standard exponential random variables.
Hence,
\begin{align*}
		\P(\tau_n <s) &= \P\left(\max_{|v|=n} J(v)< s \right) \leq \left(1- \P(W_1 >s + \gamma n) \right)^{k^n}  \leq \exp \left( - k^n \P(W_1 >s + \gamma n) \right) \\
			& =\exp (-e^{-s}).
	\end{align*}
In particular $\mathbb{P}( \tau_n < - \log(2 \log n) ) \leq 1 / n^2$ is summable in $n$, so that again by Borel-Cantelli we have 
\begin{align}\label{taularge}
		\P\big( \tau_n \geq - \log( 2 \log n) \mbox{ for all but finitely many $n$} \big) = 1 .
	\end{align}
To summarise, from \eqref{tausmall} and \eqref{taularge} we have seen that almost-surely
\begin{align} \label{eq:tau statement}
- \log (2 \log n ) \leq \tau_n \leq 2 \log n \qquad \text{for all but finitely many $n$}.
\end{align}
Let $T_n := \max_{ |v| = n } S(v)$ denote the last time at which there was an interval of size $k^{-n}$, so that $T_n =  q^{ -n} ( \tau_n + \gamma n)$. Rephrasing \eqref{eq:tau statement} we have, almost-surely,
\begin{align} \label{eq:tau statement 2}
q^{ - n} (\gamma n - \log (2 \log n )) \leq T_n \leq q^{-n} ( \gamma n + 2 \log n) \qquad \text{for all but finitely many $n$}.
\end{align}
Note that by definition $m_t := \sup \{ t \geq 0 : T_n \leq t \}$. Moreover, for $a(x) := q^{ - x} (\gamma x - \log (2 \log x )) $ and $b(x) = q^{-x} ( \gamma x + 2 \log x) $,  we are in the setting of Lemma \ref{lem:convert}, so that almost-surely there exists a $t_0 \in \mathbb{R}$ such that for all $t \geq t_0$
\begin{align} \label{eq:almost}
m_t \in \{ \lceil b^{-1}(t) \rceil , \lceil a^{-1}(t) \rceil \}.
\end{align}
It remains to obtain explicit functions from $b^{-1}(t)$ and $a^{-1}(t)$. The reader is invited to verify using the fact that $\frac{1}{\kappa} := \log\frac{1}{q}$ that with $g(t)$ as in the statement of Theorem \ref{thm:largest}, we have 
\begin{align*}
a^{ - 1}(t) = g(t) + \kappa \frac{ \log \log t }{ \log t } + o \left( \frac{ \log \log t }{ \log t } \right)
\end{align*}
and 
\begin{align*}
b^{ - 1}(t) =g(t) +  \left( \kappa - \frac{2}{ \gamma} \right) \frac{ \log \log t }{ \log t } + o \left( \frac{ \log \log t }{ \log t } \right).
\end{align*}
In particular, setting $\mu_1 := \left( \kappa + \frac{2}{ \gamma} \right)$, for all sufficiently large $t$ we have 
\begin{align*}
g(t) - \mu_1 \frac{ \log \log t}{\log t} \leq b^{ - 1}(t) \leq a^{ - 1}(t) \leq g(t) + \mu_1 \frac{ \log \log t}{\log t} .
\end{align*}
Similarly, for all sufficiently large $t$ we have $\{ \lceil b^{-1}(t) \rceil , \lceil a^{-1}(t) \rceil \} \subseteq \{ \lceil g(t) - \mu_1 \frac{ \log \log t}{\log t}  \rceil, \lceil g(t) + \mu_1 \frac{ \log \log t}{ \log t } \rceil \}$. Theorem \ref{thm:largest} now follows from \eqref{eq:almost}. \hfill $\square$
}

%%%%%%%%%%%%%%%%
\subsection{The rescaled point process} \label{sec:pp}
%%%%%%%%%%%%%%%%

We define a sequence of point processes $(N_n)_{ n \geq 1}$ on the real line as follows. The number of points $N_n(A)$ lying in a Borel set $A\subseteq\R$ is given by 
\begin{align*}
N_n(A) := \sum_{v \in \T_n} \delta_{J(v)} (A) =  \# \left\{ v \in \T_n : J(v) \in A \right\},
\end{align*}
where we recall $J(v) = K(v) - \gamma |v|$ for $\gamma = \log k $ from \eqref{def:NnAndJv}.  It follows from the linearity of expectation and Lemma \ref{lem:tails} that 
\begin{align*}
\mathbb{E} \left[ N_n([t,\infty)) \right] = k^n \mathbb{P} \left( K_n > \gamma n+ t \right) = (1 + o(1)) e^{ -t}/\varphi_\infty(q).
\end{align*}
That is, as $n$ grows, the point process $N_n$ has a unit order number of particles in each compact interval in terms of expectations. We will argue that the point process $N_n$ converges in distribution to a Poisson point process, denoted by $N_\infty$ and with intensity 
$e^{-t} /\varphi_\infty(q) dt$. Using the fact that the limiting process is simple, i.e.\, it assigns at most a unit mass to each point, we use \cite[Theorem 4.18]{kallenberg1975random} which asserts that it is sufficient to show that the avoidance functions and intensity measures converge, that is, we have for any Borel set $A \subseteq \R$,
\begin{equation}\label{eq:4:criterion}
	\P\left(N_n(A)=0\right) \to \P\left(N_\infty(A)=0\right) \quad \mbox{and} \quad \E[N_n(A)] \to \E[N_\infty(A)].
\end{equation}
We will show that \eqref{eq:4:criterion} is satisfied in our setup by using factorial measures. To introduce them we need some definitions from the theory of point processes. Following Section 4.3 of \cite{last2017lectures},
let $Y=\sum_{k}\delta_{x_k}$ be a point process on a set $E$. Define a point process $Y^{[p]}$ on $E^p$ by letting $Y^{ [p]}(A)$ be the ordered $p$-tuples of distinct points of $Y$ in $A\subseteq E^p$. 
Given a measure $\lambda$ on $E$, we define the $p^{\rm th}$ factorial measure $\lambda^{[p]}$ as follows. Suppose $Y$ is a Poisson point process on $E$ with intensity measure $\lambda$. Then
\begin{equation*}
	\lambda^{[p]}(A) = \E \left[ Y^{[p]}(A)\right], \qquad A \subseteq E^p.
\end{equation*}

%The following lemma is a paraphrasing of Remarks 1.2.3 and 1.2.4 of Hough et al \cite{hough2009zeros}.
%\begin{lemma} \label{lem:pp characterisation}
%If, for every compact subset $F$ of $E$, the random variable $Y(F)$ has exponential tails, then the law of the point process $Y$ is determined by the quantities
%\begin{align*}
%\left( \mathbb{E} \left[ Y^{ \cap p} (B) \right] : p \geq 1, B \in E^p \right).
%\end{align*}
%\end{lemma}
%We refer to the collection $\mathbb{E} \left[ Y^{ \cap p}(B) \right]$ as the \emph{joint intensities of the point process}. 
%It is easily verified that if $A$ is a cylinder set of the form $A=A_1 \times \ldots \times A_p$ the $p^{\rm th}$ factorial measure of $A$ is given via 
%\begin{align*}
%	\lambda^{[p]}(A)= \prod_{ i = 1}^p  \lambda(A_i).
%\end{align*}
%We remark that the cylinder sets $B_1 \times \ldots \times B_p$ generate the product $\sigma$-algebra on $E^k$. In particular, provided $\lambda$ is a $\sigma$-compact measure on $E$, by Lemma \ref{lem:pp characterisation} $Y$ is the unique (in distribution) point process with this property. 
In particular $\lambda^{[1]}=\lambda$. Since the sets $\{ [t, \infty) : t \in \mathbb{R} \}$ form a $\pi$-system generating the Borel subsets of $\mathbb{R}$, the Poisson process $Y$ of intensity $e^{-t}dt$ is characterised in distribution as the unique point process with the property that
\begin{align*}
\mathbb{E} \left[ Y^{[p]} \left( [t_1, \infty) \times \ldots \times [t_p, \infty) \right) \right] = \prod_{ j = 1}^p e^{ - t_j } \text{  for all $k \geq 1$ and $t_1,\ldots,t_k\in \R$.  }
\end{align*} 
The factorial measures can be used to represent the avoidance function via~\cite[formula 5.4.10]{daley2003introduction},
\begin{equation*}
	\P\left(Y(A)=0\right) = \sum_{p=0}^\infty (-1)^p \frac{\lambda^{[p]}\left(A^{(p)}\right) }{p!}, \quad A^{(p)} = \prod_{j=1}^p A, \qquad A \in E.
\end{equation*}
Therefore, provided that $\E[N_n^{[p]}(A^{(p)}] = \lambda^{[p]}(A^{[p]}) \to 0$ as $n, p \to \infty$, for~\eqref{eq:4:criterion} to hold it 
suffices to show that 
\begin{equation*}
	\lim_{n \to \infty}\mathbb{E} \left[ N_n^{[p]} \left( [t_1, \infty) \times \ldots \times [t_p, \infty) \right) \right] = \prod_{ j = 1}^p e^{ - t_j } \text{  for all $k \geq 1$ and $t_1,\ldots,t_k\in \R$.  }
\end{equation*} 
%More generally, a vector of point processes $(Y_0,\ldots,Y_{\ell-1})$ have the distribution of $\ell$ independent Poisson processes on $\mathbb{R}$ each with intensity $e^{ -t } dt$ if and only if for all reals $\left( t_{i,j} : 0 \leq i \leq \ell-1, 1 \leq j \leq p_i \right)$ we have 
%\begin{align} \label{eq:pochar}
%\mathbb{E} \left[ \prod_{ i = 0}^{\ell - 1} Y_i^{[p_i]} \left( [t_{i,1} , \infty ) \times \ldots \times [t_{i,p_i}, \infty) \right) \right] = \prod_{ i = 0}^{\ell - 1} \prod_{ j =1}^{p_i} e^{ - t_{i,j} } \text{  for all integers $p_i \geq 0$ and reals $\{ t_{i,j } \}$.}
%\end{align} 
%We permit the case that one or more of the $p_i$ may be zero in \eqref{eq:pochar}, with the empty product interpreted as being equal to $1$. 

%\begin{lemma}
%\textcolor{red}{Ask Dominik/Piotr about this: a series of point processes $Y_n$ converge in distribution to a point process $Y$ provided the joint intensities converge (provided that $Y$ has exponential tails on compacts).}  \dom{I think Nina suggested that we may want to go for the notion of vague convergence for Point processes at this point.}
%\end{lemma}
%
%%%%%%%%%%%%%%%%%
%\subsection{Main result: exit times for largest fragments}
%%%%%%%%%%%%%%%%%
We now state the main result of this section, {which rephrases Theorem \ref{thm:largestpp} in the above notation.}
\begin{thm} \label{thm:poisson}
Let $\ell$ be a positive integer. Then as $n \to \infty$, the $\ell$-tuple $(N_{n},\ldots, N_{n+\ell-1})$ of point processes converge in distribution to a $\ell$-tuple of i.i.d.\ Poisson point processes on the real line with intensity measures $e^{ - s}ds/\varphi_\infty(q)$.
\end{thm}%

By the preceding discussion, it is sufficient to show that for all non-negative integers $p_0,\ldots,p_{\ell-1}$, and all real numbers $( t_{i,j} )$ with $0 \leq i \leq \ell-1 , 1 \leq j \leq p_i$, we have 
\begin{align} \label{eq:polimits}
\lim_{ n \to \infty} \mathbb{E} \left[ \prod_{ i = 0}^{ \ell - 1}\prod_{ j =1}^{p_i} N_{n+i}^{[p_i]} \left( [t_{i,1} , \infty ) \times \ldots \times [t_{i,p_i}, \infty) \right)   \right] =  \prod_{ i = 0}^{\ell - 1} \prod_{ j =1}^{p_i} e^{ - t_{i,j}}/\varphi_\infty(q).
\end{align}
To this end, for $\vec{p} := (p_0,\ldots,p_{\ell-1})$, we define
\begin{align}\label{Tnp-def}
\T_n^{\vec{p}} := \left\{ \textbf{u} := \left(u_{i,j} : 0 \leq i \leq \ell-1, 1 \leq j \leq p_i \right) : \text{ $(u_{i,1},\ldots,u_{i,p_i})$ are distinct elements of $\T_{n+i}$} \right\}.
\end{align}
We will use the following many-to-few formula (which follows from the linearity of expectation):
\begin{align} \label{eq:mtf}
 \mathbb{E} \left[ \prod_{ i = 0}^{ \ell - 1}\prod_{ j =1}^{p_i} N_{n+i}^{[p_i]} \left( [t_{i,1} , \infty ) \times \ldots \times [t_{i,p_i}, \infty) \right)   \right]  = \sum_{  \textbf{u} \in \T_{ n }^{\vec{p}} } \mathbb{P} \left( \bigcap_{ i = 0}^{ \ell - 1 } \bigcap_{ j = 1}^{ p_i}  \{ J(u_{i,j}) > t_{i,j} \} \right).
\end{align}
We split the task of proving \eqref{eq:polimits} over the next two sections, first dealing with the easier lower bound, and then with the more difficult upper bound.

\subsection{The lower bound in \eqref{eq:polimits}}
This section is dedicated to proving the lower bound
\begin{align} \label{eq:polimitslower}
\liminf_{ n \to \infty}  \mathbb{E} \left[ \prod_{ i = 0}^{ \ell - 1}\prod_{ j =1}^{p_i} N_{n+i}^{ [p_i]} \left( [t_{i,1} , \infty ) \times \ldots \times [t_{i,p_i}, \infty) \right)   \right]  \geq \prod_{ i = 0}^{\ell - 1} \prod_{ j = 1}^{p_i} e^{ - t_{i,j}}/\varphi_\infty(q).
\end{align}
We begin with a lemma estimating the cardinality of $ \T_n^{\vec{p}} $.

\begin{lemma} \label{lem:numberoftuples}

There is a constant $C_{\vec{p}}\, \in(0,\infty)$ depending on $\vec{p} = (p_0,\ldots,p_{\ell - 1})$, but independent of $n$, such that
\begin{align*}
\prod_{ i=0}^{\ell-1} k^{ p_i(n+i) } \geq \# \T_n^{\vec{p}} \geq  \left( 1 - C_{\vec{p}} k^{ - n} \right) \prod_{ i = 0}^{\ell - 1} k^{ p_i(n+i) } .
\end{align*}  
\end{lemma}

\begin{proof}
Since
\begin{equation*}
	\# \T_n^{\vec{p}}  = \prod_{i=0}^{\ell-1} {k^{n+i} \choose p_i}
\end{equation*}
the desired bounds follow by applying the Stirling estimates $\sqrt{ 2 \pi } m^{m+1/2} e^{ - m} \leq m! \leq e m^{m+1/2} e^{-m}$.
%The upper bound follows from the fact that $\prod_{ 0 \leq i \leq \ell-1} k^{ p_i(n+i) } $ is the cardinality of collection of set of (possibly non-distinct) tuples $(v_{i,j})_{i,j}$ where $v_{i,j}$ is an element of $V_{n+i}$. 
%
%To prove the lower bound, consider the subset $\bar{V}_{\langle n \rangle}^{(p)}$ of $V_{ \langle n \rangle}$ consisting of tuples $(u_{i,j})$ with $p = p_1 + \ldots + p_n$ distinct ancestors in the time $n$ generation. Then $\bar{V}_{ \langle n \rangle}^{(p)}$ has cardinality $k(k - 1)\ldots(k - p + 1) \prod_{ i = 1}^n k^{ i p_i} \geq  \left( 1 - C(p) k^{ - n} \right) \prod_{ i = 0}^{\ell - 1} k^{ p_i(n+i) } $
%for some constant $C(p)$. Since the cardinality of  $V^{(p)}_{ \langle n \rangle}$ is at least as large as that of $\bar{V}_{\langle n \rangle}^{(p)}$, the lower bound is proved. 
\end{proof}

The following lemma is an FKG type inequality for correlated events on the tree. Since the proof is not related to the rest of our arguments, it will be given in the appendix.

\begin{lemma}\label{lem:fkg}
For any $n \in \N$, $\vec{p} = (p_0,\ldots,p_{\ell - 1})$ and any tuple $\textbf{u} := (u_{i,j} : 0 \leq i \leq \ell-1, 1 \leq j \leq p_i )$, we have 
\begin{align} \label{eq:red1}
\mathbb{P} \left( \bigcap_{ i = 0}^{ \ell - 1 } \bigcap_{ j = 1}^{ p_i}  \{ J(u_{i,j}) > t_{i,j} \} \right) \geq \prod_{ i = 0}^{ \ell - 1} \prod_{ j = 1}^{ p_i}  \mathbb{P} \left( J(u_{i,j}) > t_{i,j} \right)
\end{align}
for any choice of real numbers $(t_{i,j})$.
%Better without indizes!!
\end{lemma}

With Lemma \ref{lem:fkg} at hand, we are now ready to prove the lower bound \eqref{eq:polimitslower}.

\begin{proof}[Proof of \eqref{eq:polimitslower}]
By Lemma \ref{lem:tails}, recalling $J(v) = K(v) - \gamma |v|$ from \eqref{def:NnAndJv}, we have for all $i,j$ that
\begin{align} \label{eq:orange1}
 \mathbb{P} \left( J(u_{i,j}) > t_{i,j} \right)  \geq \frac{ e^{ - t_{i,j} - \gamma ( n + i ) }}{ \varphi_{n+i}(q) } \left( 1 - C_{q} e^{ - (q^{-1} - 1) \left( \gamma n + t_{0}\right) }   \right),
\end{align} 
where $t_0 := \min_{i,j} \{ t_{i,j} \} $ and $C_q\in(0,\infty)$ is some constant depending on $q$, but independent of $n$ and  $(t_{i,j})$. Now, for a sufficiently large constant $C_{\vec{p}\, ,q, t_0}\in(0,\infty)$ depending on $q$, $t_0$ and $\vec{p} = (p_0,\ldots,p_{\ell-1})$, but again independent of $n$ and  $(t_{i,j})$, we have by the Bernoulli inequality for sufficiently large $n$,
\begin{align} \label{eq:yellow1} 
\prod_{ i = 0}^{ \ell - 1} \prod_{ j = 1}^{ p_i}  \left( 1 - C_q e^{ - (q^{-1} - 1) \left( \gamma n + t_{0}\right) }  \right) =  \left( 1 - C_q e^{ - (q^{-1} - 1) \left( \gamma n + t_{0}\right) }  \right)^{|\vec{p}\,|} \geq 1 - C_{q,\vec{p},\,t_0} e^{ - ( q^{-1} - 1)  \gamma n  },
\end{align}
where $|\vec{p}\, | = p_0+p_1+\ldots +p_{\ell-1}$.
Combining \eqref{eq:red1},  \eqref{eq:orange1} and \eqref{eq:yellow1}, for any tuple $(u_{i,j})$ in $\T_n^{\vec{p}}$, we have 
\begin{align*} %\label{eq:lowerprob}
\mathbb{P} \left( \bigcap_{ i = 0}^{\ell -1 } \bigcap_{ j =1}^{ p_i}  \{ J(u_{i,j}) > t_{i,j} \} \right) \geq  \left(  1 - C_{q,\vec{p},\, t_0} e^{ - ( q^{-1} - 1)  \gamma n  } \right) \prod_{ i = 0}^{\ell - 1} \prod_{ j = 1}^{p_i} e^{ - t_{i,j} - \gamma(n+i) }/\varphi_{n+i}(q).
\end{align*}
In particular, by plugging this in the  many-to-few formula \eqref{eq:mtf}, we obtain
\begin{align*} %\label{eq:green1}
 &\mathbb{E} \left[ \prod_{ i = 0}^{ \ell - 1}\prod_{ j =1}^{p_i} N_{n+i}^{[p_i]} \left( [t_{i,1} , \infty ) \times \ldots \times [t_{i,p_i}, \infty) \right)   \right] \\& \geq \# \T_n^{\vec{p}} \left(  1 -  C_{q,\vec{p},\, t_0} e^{ - ( q^{-1} - 1)  \gamma n  }  \right) \prod_{ i = 0}^{\ell - 1} \prod_{ j = 1}^{p_i} e^{ - t_{i,j} - \gamma(n+i) }/\varphi_{n+i}(q).
\end{align*}
Now using the fact that $\gamma = \log k $, and the lower bound in Lemma \ref{lem:numberoftuples}, we find that
\begin{align*} %\label{eq:blue1}
\# \T_n^{\vec{p}} \prod_{ i = 0}^{\ell - 1} \prod_{ j = 1}^{p_i} e^{ - \gamma(n+i)} \geq (1 - C_{\vec{p}}k^{-n} ).
\end{align*}
Combining the last two estimates we obtain
\begin{align*}
 &\mathbb{E} \left[ \prod_{ i = 0}^{ \ell - 1}\prod_{ j =1}^{p_i} N_{n+i}^{[p_i]} \left( [t_{i,1} , \infty ) \times \ldots \times [t_{i,p_i}, \infty) \right)   \right] \\& \geq (1 - C_{\vec{p}}\, k^{-n} ) \left(  1 -  C_{q,\vec{p},\, t_0} e^{ - ( q^{-1} - 1)  \gamma n  }  \right) \prod_{ i = 0}^{\ell - 1} \prod_{ j = 1}^{p_i} e^{ - t_{i,j} }/\varphi_{n+i}(q).
\end{align*}
Taking $n \to \infty$ concludes the proof of \eqref{eq:polimitslower}.
\end{proof}

%%%%%%%%%%%%%%%%
\subsection{The hard direction in Theorem \ref{thm:largestpp}: an overview}
%%%%%%%%%%%%%%%%

In this section we work towards proving the upper bound in \eqref{eq:polimits}. Namely, the goal is to show that  
\begin{align} \label{eq:upperlim}
\limsup_{ n \to \infty} \mathbb{E} \left[ \prod_{ i = 0}^{ \ell - 1}\prod_{ j =1}^{p_i} N_{n+i}^{ [p_i]} \left( [t_{i,1} , \infty ) \times \ldots \times [t_{i,p_i}, \infty) \right)   \right] \leq \prod_{ i = 0}^{\ell - 1} \prod_{ j = 1}^{p_i} e^{ - t_{i,j}}/\varphi_\infty(q).
\end{align}
In light of the many-to-few formula \eqref{eq:mtf}, to tackle 
the hard direction, we need for all tuples $(u_{i,j})$ in $\T_n^{\vec{p}}$ to obtain effective upper bounds on the \emph{exceedance probabilities} 
\begin{align*}
\mathbb{P} \left(\bigcap_{ i = 0}^{\ell -1 } \bigcap_{ j =1}^{ p_i}   \{J(u_{i,j}) > t_{i,j} \}  \right) .
\end{align*}

We now overview the main idea in proving an inequality of the form \eqref{eq:upperlim}. Given integers $0 \leq m \leq n$ and a tuple $\textbf{u} = \left( u_{i,j} : 0 \leq i \leq \ell -1, 1 \leq j \leq p_i \right)$, we define the number $1 \leq P_{n-m}( \textbf{u)} \leq |\vec{p}\, |=p_0 + \ldots + p_{ \ell - 1}$ by setting
\begin{align}\label{pnm-def}
P_{n-m}(\textbf{u}) := \text{number of different ancestors in generation $n-m$ of the vertices $u_{i,j}$}.
\end{align}
For a special choice of $m$ which we give below, we distinguish between two different types of tuples:

\begin{itemize}
\item We say a tuple $\textbf{u}$ in $\T_n^{\vec{p}}$ is \emph{distantly related} (in generation $n-m$) if 
\begin{align*}
P_{n-m}( \textbf{u} ) = p_0 + \ldots + p_{\ell-1} = |\vec{p}\, |.
\end{align*}
We will see that provided $m \leq \theta n$ for some constant $\theta <1$, the overwhelming number of tuples in $\T_n^{\vec{p}}$ are distantly related as $m$ and $n$ become large. Since the bulk of particles are of this form, we will require a fairly delicate $m$-dependent control on the exceedance probabilities; see Lemma \ref{lem:distant} below. 
\item We say a tuple $\textbf{u}$ in $\T_n^{\vec{p}}$ is \emph{$\nu$-closely related} (in generation $n-m)$ if 
\begin{align*}
P_{n-m}( \textbf{u} ) = \nu \text{ for some $\nu < |\vec{p}\, |$}.
\end{align*}
We find that for such tuples, the exceedance probability $\mathbb{P} \left(\bigcap_{ i = 0}^{\ell -1 } \bigcap_{ j =1}^{ p_i}   \{J(u_{i,j}) > t_{i,j} \}  \right)$ has a larger order than for distantly related tuples. 
However, it turns out that this order is negligible when compared with the relative size of the number of closely related tuples. Indeed, we show in Lemma \ref{lem:count} that the number of tuples $\textbf{u}$ with $P_{n-m}(\textbf{u}) = \nu$ has the order $k^{ \nu n}$, while Lemma \ref{lem:close} tells us that the associated exceedance probabilities are of order $o(k^{- \nu n})$. 
\end{itemize}

The next two Lemmas are the main results of this section, controlling respectively the  exceedance probabilities associated with distantly and closely related tuples.

\begin{lemma} \label{lem:distant}
Let $\textbf{u}$ be distantly related. Then for all $m\in\N$ such that $|\vec{p}\,|q^{m-1} \leq 1/2$, we have 
\begin{align} \label{eq:fine}
\mathbb{P} \left(\bigcap_{ i = 0}^{\ell -1 } \bigcap_{ j =1}^{ p_i}   \{J(u_{i,j}) > t_{i,j} \}  \right)  \leq (1 + \varepsilon_{n,m}) \prod_{ i = 0}^{\ell - 1} \prod_{ j = 1}^{p_i} \left(e^{ - \gamma(n+i) - t_{i,j} }/\varphi_{\infty}(q) \right),
\end{align}
for all real numbers $t_{i,j}$, where 
\begin{align*}
\varepsilon_{n,m} = C_{\vec{p},\, q, t_0} \left( e^{ - \frac{1}{q} \gamma n  } + q^m \right)
\end{align*}
for $t_0 := \min_{i,j} \{ t_{i,j} \}$.
\end{lemma}

\begin{lemma} \label{lem:close}
Let $m\in\N$ be sufficiently large so that $2|\vec{p}\, | q^m < 1/2$. Let $t_0 := \min_{ i,j} \{ t_{i,j} \}$ for real numbers $t_{i,j}$, and $\textbf{u}$ in $\T_n^{\vec{p}}$ be such that $P_{n-m}( \textbf{u}) = \nu$. Then  $\mathbb{P}  \left(\bigcap_{ i = 0}^{\ell -1 } \bigcap_{ j =1}^{ p_i}   \{J(u_{i,j}) > t_{i,j} \}  \right)  = o( k^{ - \nu n })$. More specifically, there is a constant $C_{q,t_0}\in(0,\infty)$ such that 
\begin{align*}
\mathbb{P}  \left(\bigcap_{ i = 0}^{\ell -1 } \bigcap_{ j =1}^{ p_i}   \{J(u_{i,j}) > t_{i,j} \}  \right)  \leq C_{q,t_0} \exp \left(  - (\theta_q + \nu ) \gamma n \right),
\end{align*}
where $\theta_q := \min \left\{ \frac{1}{q} , 2 - q \right\} - 1 > 0$. 
\end{lemma}

The proofs of Lemma \ref{lem:close} and Lemma \ref{lem:distant} are both lengthy, and we defer them to Sections \ref{sec:Distantly} and \ref{sec:Closely}, respectively. We now conclude this overview  with the following short lemma on the number of closely related tuples, which will be used in conjunction with Lemma \ref{lem:close} and Lemma \ref{lem:distant} to prove the upper bound \eqref{eq:upperlim}.

\begin{lemma} \label{lem:count}
Recall \eqref{Tnp-def} and \eqref{pnm-def}. We have the following bound on the number of $\nu$-closely related tuples in generation $n-m$:
\begin{align*}
\# \left\{ \textbf{u} \in \T_n^{\vec{p}} : P_{n-m} ( \textbf{u} ) = \nu \right\} \leq C_{k,\ell,\vec{p}}\, k^{ n \nu} k^{ (|\vec{p}\,| - \nu ) m }.
\end{align*}
\end{lemma}
\begin{proof}
There are at most $k^{(n-m) \nu}$ ways of choosing $\nu$ different ancestors in generation $n-m$. Each individual in generation $n-m$ has $k^{m + i }$ descendents in generation $n+i$. Using the crude bound that $k^{ m + i} \leq k^{ m + \ell - 1}$ whenever $i \leq \ell - 1$, the number of $\nu$-closely related tuples in generation $n-m$ is bounded from above by
\begin{align*}
k^{ (n-m) \nu } \cdot \left( k^{ m + \ell - 1 } \right)^{ |\vec{p}\,| } = C_{k,\ell,\vec{p}} \, k^{ \nu n } k^{ (|\vec{p}\,| - \nu ) m },
\end{align*}
where $C_{k,\ell,\vec{p}} := k^{ |\vec{p}\,| ( \ell-1) }$.
 \end{proof}

We  now  show how Lemma \ref{lem:distant}, Lemma \ref{lem:close} and Lemma \ref{lem:count} are combined to obtain  \eqref{eq:upperlim}.

\begin{proof}[Proof of \eqref{eq:upperlim} assuming Lemma \ref{lem:distant} and Lemma \ref{lem:close}]
By the many-to-few formula \eqref{eq:mtf}, we have 
\begin{align} \label{eq:truerep}
 &\mathbb{E} \left[ \prod_{ i = 0}^{ \ell - 1}\prod_{ j =1}^{p_i} N_{n+i}^{ [p_i]} \left( [t_{i,1} , \infty ) \times \ldots \times [t_{i,p_i}, \infty) \right)   \right] \nonumber \\
&= \sum_{ \textbf{u} : P_{n-m}( \textbf{u}) = |\vec{p}\,|} \mathbb{P}  \left(\bigcap_{ i = 0}^{\ell -1 } \bigcap_{ j =1}^{ p_i}   \{J(u_{i,j}) > t_{i,j} \}  \right)  + \sum_{ \nu = 1}^{ |\vec{p}\,|- 1} \sum_{ \textbf{u} : P_{n-m}(\textbf{u}) = \nu} \mathbb{P} \left(\bigcap_{ i = 0}^{\ell -1 } \bigcap_{ j =1}^{ p_i}   \{J(u_{i,j}) > t_{i,j} \}  \right)  .
\end{align}
We begin by controlling the contribution from distantly related tuples. Since there are at most $\prod_{ i = 0}^{\ell - 1} k^{ p_i(n+i) }$ elements in 
$\T_n^{\vec{p}}$, using Lemma \ref{lem:distant}, we obtain
\begin{align} \label{eq:sibelius1}
 \sum_{ \textbf{u} : P_{n-m}( \textbf{u}) = |\vec{p}\,|} \mathbb{P}  \left(\bigcap_{ i = 0}^{\ell -1 } \bigcap_{ j =1}^{ p_i}   \{J(u_{i,j}) > t_{i,j} \}  \right)  & \leq  (1 + \varepsilon_{n,m})  \prod_{ i = 0}^{\ell - 1} k^{ p_i(n+i) }\prod_{ i = 0}^{\ell - 1} \prod_{ j = 1}^{p_i} \big( e^{ - \gamma(n+i) - t_{i,j} } / \varphi_{\infty}(q) \big) \nonumber \\
&=  (1 + \varepsilon_{n,m})  \prod_{ i = 0}^{\ell - 1} \prod_{ j = 1}^{p_i} \big( e^{ - t_{i,j} }/\varphi_{\infty}(q) \big) \nonumber \\
&\leq  \prod_{ i = 0}^{\ell - 1} \prod_{ j = 1}^{p_i} \big( e^{ - t_{i,j} }/\varphi_{\infty}(q) \big) + \varepsilon_{n,m} \big( e^{ -t_0 }/\varphi_\infty(q)\big) ^{- |\vec{p}\,|} ,
\end{align}
where $\varepsilon_{n,m}$ is as in the statement of Lemma \ref{lem:distant}, i.e.\, $\varepsilon_{n,m} = C_{\vec{p}, \, q, t_0} ( e^{ - \frac{1}{q} ( \gamma n -t_0) } + q^m )$, and $t_0 = \min_{i,j} \{ t_{i,j} \}$. We now control the contribution from closely related tuples. Indeed, combining Lemma~\ref{lem:close} with Lemma~\ref{lem:count}, provided that $2|\vec{p}\,| q^m < 1/2$, for each $1 \leq \nu \leq |\vec{p}\,|-1$, we have
\begin{align} \label{eq:sibelius2}
\sum_{ \textbf{u} : P_{n-m}(\textbf{u}) = \nu} \mathbb{P} \left(\bigcap_{ i = 0}^{\ell -1 } \bigcap_{ j =1}^{ p_i}   \{J(u_{i,j}) > t_{i,j} \}  \right)  &\leq C_{k, \ell,\vec{p},\, q} k^{n \nu} k^{ (|\vec{p}\,|- \nu)m } \exp \left( - \theta_q \gamma n  - \nu \gamma n \right) \nonumber \\
& \leq C_{k, \ell,\vec{p},\, q}  \exp \left( \gamma( (|\vec{p}\,| - 1 )m - \theta_q n ) \right).
\end{align} 
Now for each $n\in\N$, we set $m := \tilde{\theta}_q n$, where $|\vec{p}\,| \tilde{\theta}_q < \theta_q$, and send $n \to \infty$. Now by using the bounds \eqref{eq:sibelius1} and \eqref{eq:sibelius2} in \eqref{eq:truerep}, we obtain \eqref{eq:upperlim}. 
\end{proof}

This finishes the proof of the upper bound \eqref{eq:upperlim}, and thereby completing the proof of Theorem \ref{thm:poisson}, {respectively Theorem \ref{thm:largestpp}.} It remains to prove Lemma \ref{lem:distant} and Lemma \ref{lem:close}, which we do in the next two sections.

%%%%%%%%%%%%%%%%%%%%%%%%%%%%%
\subsection{Bounding exceedance probabilities of distantly related tuples}\label{sec:Distantly}
%%%%%%%%%%%%%%%%%%%%%%%%%%%%%

\begin{proof}[Proof of Lemma \ref{lem:distant}]
Let $P_{n-m}( \textbf{u}) = |\vec{p}\,|$. For each $i,j$, let $v_{i,j}$ denote the ancestor of $u_{i,j}$ in generation $(n-m)$. Since $P_{n-m}( \textbf{u} ) = | \vec{p}\, |$, the sites $v_{i,j}$ form $|\vec{p}\, |$ distinct elements of generation $n-m$. Now by construction, we have
\begin{align} \label{eq:linz0} 
\mathbb{P}  \left(\bigcap_{ i = 0}^{\ell -1 } \bigcap_{ j =1}^{ p_i}   \{ J(u_{i,j}) > t_{i,j} \}  \right) &= \mathbb{P}  \left(\bigcap_{ i = 0}^{\ell -1 } \bigcap_{ j =1}^{ p_i}   \{K(u_{i,j}) > t_{i,j} + \gamma(n+i)  \}  \right)  \nonumber  \\
&= \mathbb{P}  \left(\bigcap_{ i = 0}^{\ell -1 } \bigcap_{ j =1}^{ p_i}   \{ q^{m+i} K(v_{i,j}) + K_{m+i}^{(i,j)} > t_{i,j} + \gamma(n+i) \}  \right),
\end{align}
where the variables $\left\{ K_{m+i}^{(i,j)}: 0 \leq i \leq \ell-1 , 1 \leq j \leq p_i \right\}$ are independent, and each $K_{m+i}^{(i,j)}$ is distributed as $K_{m+i}$.

Consider the following general fact. If for a finite indexing set $\mathcal{E}$, $(A_e )_{ e \in \mathcal{E}}$ are identically distributed (and possibly dependent) random variables with the same law as $A$, and $(B_e )_{ e \in \mathcal{E}}$ are (possibly dependent but) independent of $(A_e )_{ e \in \mathcal{E}}$ with any distributions, then we have that 
\begin{align}\label{replacebyone}
\mathbb{P} \left( \bigcap_{ e \in \mathcal{E}} \{ A_e + B_e > c_e \} \right) \leq \mathbb{P} \left( \bigcap_{e \in \mathcal{E} } \{ A + B_e > c_e \} \right).
\end{align} 
The proof is easy and can be found in \cite[Lemma 5.2]{gantert2018large}. Using this fact in \eqref{eq:linz0} with $A_{i,j} = K(v_{i,j})$ and $B_{i,j} = K_{m+i+1}^{(i,j)}$, we may replace $K(v_{i,j})$ in \eqref{eq:linz0} with a single copy of $K_{n-m}$, so that we have the upper bound
\begin{align} \label{eq:linz1}
\mathbb{P}  \left(\bigcap_{ i = 0}^{\ell -1 } \bigcap_{ j =1}^{ p_i}   \{ J(u_{i,j}) > t_{i,j} \}  \right) &\leq \mathbb{P}  \left(\bigcap_{ i = 0}^{\ell -1 } \bigcap_{ j =1}^{ p_i}   \{ q^{m+i} K_{n-m} + K_{m+i}^{(i,j)} > t_{i,j} + \gamma(n+i) \}  \right).
\end{align}
Finally, using the fact that $K_\infty$ stochastically dominates $K_n$ for each $n\in\N$, as well as the fact that $q^{m+i} \leq q^m$, we may simplify several matters of indexing by extracting from \eqref{eq:linz1} the upper bound
\begin{align} \label{eq:linz2}
\mathbb{P}  \left(\bigcap_{ i = 0}^{\ell -1 } \bigcap_{ j =1}^{ p_i}   \{ J(u_{i,j}) > t_{i,j} \}  \right) &\leq \mathbb{P}  \left(\bigcap_{ i = 0}^{\ell -1 } \bigcap_{ j =1}^{ p_i}   \{ q^{m+i} K_{\infty} + K_{\infty}^{(i,j)} > t_{i,j} + \gamma(n+i) \}  \right) \nonumber \\
&= \mathbb{E} \left[  \prod_{ i = 0}^{ \ell-1} \prod_{ j = 1}^{p_i} \mathbb{P} \left( K_\infty^{(i,j)} > t_{i,j} + \gamma(n+i) - q^m K_\infty \bigg| K_\infty \right) \right],
\end{align}
where $K_\infty$ and $K_\infty^{(i,j)}$ are independent copies of $K_\infty$, and the last equality above follows from using the definition of conditional expectation. We now control the terms inside the product in the expectation on the right-hand side of \eqref{eq:linz2}. Indeed, by
Lemma~\ref{lem:tails} for each $i,j$ we have 
\begin{align*}
\mathbb{P} \left( K_\infty^{(i,j)} > t_{i,j} + \gamma(n+i) - q^{m}K_\infty \bigg| K_\infty \right) \leq \frac{ e^{ - (t_{i,j} + \gamma(n+i) - q^m K_\infty) } }{ \varphi_\infty(q) } \left( 1 + C_q e^{ - (1/q - 1) (t_0 + \gamma n - q^{m} K_\infty )_+ } \right),
\end{align*} 
where we recall $t_0 := \min_{i,j} \{ t_{i,j}  \}$, and for a real number $x$, we let $x_+$ denote the maximum of $x$ and $0$. 

Now for every $n \in \mathbb{N}$ and every $C_q \in(0,\infty)$, there is a second constant $C_{q,n}\in(0,\infty)$ such that $(1 + C_q w )^n \leq 1 + C_{q,n} w$ for all $w \in [0,1)$. In particular, setting
$w =  e^{ - (1/q - 1) (t_0 + \gamma n - q^{m} K_\infty )_+ }$ we have 
\begin{align} \label{eq:salzburg}
& \prod_{ i = 0}^{ \ell-1} \prod_{ j = 1}^{p_i} \mathbb{P} \left( K_\infty^{(i,j)} > t_{i,j} + \gamma(n+i) - q^{m}K_\infty \bigg| K_\infty \right) \nonumber \\
& \leq \left(1 + C_{\vec{p},\, q,t_0} e^{ -(1/q-1)(\gamma n - q^{m} K_\infty)  } \right) e^{ |\vec{p}\,|q^{m}K_\infty}  \prod_{ i = 0}^{ \ell-1} \prod_{ j = 1}^{p_i} \big( e^{ - t_{i,j} - \gamma(n+i)} / \varphi_\infty(q)\big) .
\end{align}
Plugging \eqref{eq:salzburg} into \eqref{eq:linz2}, we obtain
\begin{align} \label{eq:graz}
&\mathbb{P}  \left(\bigcap_{ i = 0}^{\ell -1 } \bigcap_{ j =1}^{ p_i}   \{J(u_{i,j}) > t_{i,j} \}  \right) \nonumber \\
&\leq \left(   \mathbb{E} \left[ e^{ |\vec{p}\,|q^{m}K_\infty} \right]   + C_{\vec{p},\, q,t_0} e^{ -(1/q-1)\gamma n }  \mathbb{E} \left[ e^{ ( |\vec{p}\,|q^{m} + (1/q - 1)q^{m} ) K_\infty} \right]  \right) \prod_{ i = 0}^{ \ell-1} \prod_{ j = 1}^{p_i}  \big(e^{ - t_{i,j} - \gamma(n+i)} / \varphi_\infty(q) \big)\nonumber\\
&\leq \left( 1    + C_{\vec{p},\, q,t_0} e^{ -(1/q-1)\gamma n }  \right) \mathbb{E} \left[ e^{  |\vec{p}\,| q^{m-1} K_\infty} \right]  \prod_{ i = 0}^{ \ell-1} \prod_{ j = 1}^{p_i} \big( e^{ - t_{i,j} - \gamma(n+i)} / \varphi_\infty(q) \big),
\end{align}
where the final inequality above follows from the fact that both $|\vec{p}\,| q^{ m }$ and $|\vec{p}\,|q^{m} + (1/q - 1)q^{ m}$ are bounded from above by $|\vec{p}\,| q^{m-1}$. Now by Lemma \ref{lem:tails}, there is a constant $C_q\in(0,\infty)$ such that whenever $\theta \leq 1/2$, we have \begin{equation*}
	\mathbb{E} \left[ e^{ \theta K_\infty } \right] \leq 1 + C_q \theta .
\end{equation*} 
In particular, provided that $m\in\N$ is sufficiently large so that $|\vec{p}\, |q^{m-1} \leq 1/2$, using this inequality in \eqref{eq:graz}, we obtain
\begin{align*}
&\mathbb{P}  \left(\bigcap_{ i = 0}^{\ell -1 } \bigcap_{ j =1}^{ p_i}   \{J(u_{i,j}) > t_{i,j} \}  \right)  \\
&\leq \left( 1    + C_{\vec{p},\, q,t_0} e^{ -(1/q-1)\gamma n } \right) \left( 1  + C_{q} |\vec{p}\,| q^{m-1} \right)  \prod_{ i = 0}^{ \ell-1} \prod_{ j = 1}^{p_i}  \big(e^{ - t_{i,j} - \gamma(n+i)} / \varphi_\infty(q)\big)\\
&\leq ( 1 + \varepsilon_{n,m} ) \prod_{ i = 0}^{ \ell-1} \prod_{ j = 1}^{p_i} \big( e^{ - t_{i,j} - \gamma(n+i)} / \varphi_\infty(q)\big),
\end{align*} 
where $\varepsilon_{n,m} = C_{\vec{p},\, q,t_0}\left( e^{ - (1/q - 1) \gamma n } + q^{m-1} \right)$ for a sufficiently large constant $C_{\vec{p},\, q,t_0}\in(0,\infty)$.
\end{proof}

%%%%%%%%%%%%%%%%%%%%%%%%%%%%%
\subsection{Bounding exceedance probabilities of closely related tuples}\label{sec:Closely}
%%%%%%%%%%%%%%%%%%%%%%%%%%%%%
\begin{lemma} \label{lem:2bound}
Let $w$ and $w'$ be distinct elements in $\T = \bigcup_{ n \in \N} \T_n$. Then there exists a constant $C_q\in(0,\infty)$ such that for all $L>0$
\begin{align*}
\mathbb{P} \left( K(w) > L , K(w') > L \right) \leq C_q \exp \left( - \lambda_q L \right),
\end{align*}
where $\lambda_q :=\min \left\{ \frac{1}{q} , 2 - q \right\}$.
\end{lemma}

\begin{proof}
Let $v=w \wedge w'$ be the most recent common ancestor of $w$ and $w'$, so that $v$ is in generation $n$ and $w$ and $w'$ are in generations $n + c$ and $n + c'$ respectively. Since $w \neq w'$, we have $\max\{c, c'\} \geq 1$. Without loss of generality, we can assume that $c'\geq 1$. Then, taking into account \eqref{recrepofK},
\begin{align*}
\mathbb{P} \left( K(w) > L , K(w')  > L \right) &= \mathbb{P} \left( q^c K_n + \tilde{K}_c > L , q^{c '} K_n + \bar{K}_{c'} > L \right),
\end{align*}
where $\tilde{K}_c$, $\bar{K}_{c'}$ and $K_n$ are independent, and $\tilde{K}_c$ is distributed as $K_c$, $\bar{K}_{c'}$ is distributed as $K_{c'}$.
Taking a rather generous bound using the facts that $q \leq 1$,  and that $K_n$ is stochastically dominated by $K_\infty$, we  have
\begin{align*}
\mathbb{P} \left( K(w) > L , K(w')  > L \right) \leq \mathbb{P} \left(  K_\infty + \tilde{K}_\infty > L , qK_\infty  + \bar{K}_\infty > L \right),
\end{align*} 
where $K_\infty, \tilde{K}_\infty$ and $ \bar{K}_\infty$ are i.i.d., recalling \eqref{def:K_infty}.
By conditioning on the value of $K_\infty$, we have 
\begin{align*}
\mathbb{P} \left(  K_\infty + \tilde{K}_\infty > L , qK_\infty  + \bar{K}_\infty > L \right)= \int_0^\infty f_\infty(s) \mathbb{P} \left ( K_\infty > L - q s \right)\mathbb{P} \left ( K_\infty > L - s \right) \mathrm{d} s.
\end{align*}
Due to \eqref{crude} and Lemma \ref{lem:tails}, there is a constant $C_q\in(0,\infty)$ such that $f_\infty(s) \leq C_q e^{ - s}$ and $\mathbb{P} \left( K_\infty > M \right) \leq C_q e^{ - M_+ }$, and we obtain 
\begin{align*}
\mathbb{P} \left( K_\infty + K'_\infty > L , qK_\infty  + K_\infty'' > L \right) &\leq C_q  \int_0^\infty \exp \left( - s -  (L - qs)_+ -(L-s)_+ \right) \mathrm{d} s\\
 &\leq C_q \left[ \int_0^{L} e^{  - 2L+qs } \mathrm{d} s + \int_L^{L/q} e^{ - s - (L-qs) } \mathrm{d} s + \int_{ L/q}^\infty e^{ - s} \mathrm{d} s \right]\\
&\leq C_q \exp \left( - \min \left\{ \frac{1}{q} , 2 - q \right\} L \right)
\end{align*}
for a sufficiently large constant $C_q\in(0,\infty)$. This proves the claim.
\end{proof}

\begin{proof}[Proof of Lemma \ref{lem:close}]
Let $\textbf{u}= \left(u_{i,j} : 0 \leq i \leq \ell - 1, 1 \leq j \leq p_i \right)$ be a $\nu$-closely related tuple in generation $n-m$. Since $\nu < p_0 + \ldots + p_{ \ell - 1}$, by the pigeonhole principle, there exists an element $v_0$ of generation $n-m$ that has more than one descendent among the set $\left\{ u_{i,j} : 0 \leq i \leq \ell -1 , 1 \leq j \leq p_i \right\}$. Let $(w,w')$ be any two distinct members of the tuple $\textbf{u}$ that are descendants of $v_0$. Let $\{v_1,\ldots,v_{\nu-1}\}$ be the other $\nu - 1$ ancestors of $\textbf{u}$ in generation $n-m$, and for each $1 \leq i \leq \nu-1$, let $w_i$ be an 
element of $\textbf{u}$ which has ancestor $v_i$.
Define integers $s,s',s_1,\ldots,s_{\nu-1} \in \{0,1,\ldots, \ell - 1\}$ to be the generations such that $w \in \T_{n+s}, w' \in \T_{n+s'}, w_i \in \T_{n+s_i}$. 
For $t_0 = \min \{ t_{i,j} \}$, we have the simple relation
\begin{align*} 
\bigcap_{ i = 0}^{\ell -1 } \bigcap_{ j =1}^{ p_i}   \{J(u_{i,j}) > t_{i,j} \}  \subseteq \{ J(w) > t_0 \} \cap \{ J(w') > t_0 \} \cap \bigcap_{ i = 1}^{\nu - 1} \{ J(w_i) > t_0 \} =: A_0 \cap\bigcap_{ i = 1}^{\nu - 1} A_i,
\end{align*}
where $A_0 :=  \{ J(w) > t_0 \} \cap \{ J(w') > t_0 \} $ and $A_i := \{ J(w_i) > t_0 \}$. In particular, we have the rather generous upper bound on the exceedance probability
\begin{align} \label{eq:genrep}
\mathbb{P} \left( \bigcap_{ i = 0}^{\ell -1 } \bigcap_{ j =1}^{ p_i}   \{J(u_{i,j}) > t_{i,j} \} \right)  \leq \mathbb{P} \left( A_0 \cap\bigcap_{ i = 1}^{\nu - 1} A_i \right).
\end{align}
For integers $N$, let $\mathcal{F}_N := \sigma \left( J(v) : v \in \T_i , i \leq N \right)$. We note that the events $A_0,\ldots,A_{\nu-1}$ are conditionally independent given $\mathcal{F}_{n-m}$, each $A_i$ conditionally depending only on $J(v_i) = K(v_i) - \gamma |v_i |$. In particular,
\begin{align} \label{eq:erep}
\mathbb{P} \left( A_0 \cap\bigcap_{ i = 1}^{\nu - 1} A_i \right) = \mathbb{E} \left[ \mathbb{P} \left( \left. A_0 \cap\bigcap_{ i = 1}^{\nu - 1} A_i \right| \mathcal{F}_{n-m} \right) \right] =   \mathbb{E} \left[ \prod_{ i = 0}^{\nu - 1} \psi_i \left( K(v_i) \right) \right],
\end{align}
where $\psi_i (x) := \mathbb{P} \left( A_i | K(v_i) = x \right)$. We now obtain effective upper bounds on the functions $\psi_i(x)$, first looking at the case $i \geq 1$, and then treating the $i = 0$ case separately.

For $i \geq 1$, using the definition of $(J(v))_{v \in \T}$ for the second equality below, 
\begin{align*}
\psi_i(x) &:=\mathbb{P} \left( J(w_i) > t_0 | K(v_i) = x \right)\\
&= \mathbb{P} \left( K(w_i) > t_0 + \gamma(n + s_i ) | K(v_i) = x \right)\\
& \leq  \mathbb{P} \left( K(w_i) > L_0 | K(v_i) = x \right),
\end{align*}
where $L_0 := t_0 + \gamma n \leq  t_0 + \gamma(n + s_i ) $. Now, continuing this calculation, we use the definition of the rescaled expanding branching random walk $(K(v))_{v \in \T}$ to obtain the equality below.  We take further generous bounds to obtain the following inequality in the second line below, and then the tail bound Lemma~\ref{lem:tails} to obtain the inequality in the third line below, yielding 
\begin{align*}
\mathbb{P} \left( K(w_i) > L_0 | K(v_i) = x \right) &= \mathbb{P} \left( q^{ s_i + m } x + K_{s_i + m} > L_0 \right)\\
&\leq \mathbb{P} \left(  K_\infty > L_0 - q^m x \right) \\
&\leq C_q \exp \left( - (L_0 - q^m x ) \right)\frac{1}{\varphi_\infty(q)}.
\end{align*}
In summary, for each $1 \leq i \leq \nu - 1$ we have 
\begin{align} \label{eq:bound1}
\psi_i(x) \leq C_q \exp \left( - (L - q^m x ) \right) ,
\end{align}
where $L = t_0 + \gamma n - \log \varphi_\infty(q)> L_0$. 
We now turn to estimating $\psi_0(x)$ using Lemma \ref{lem:2bound}. Indeed, using Lemma \ref{lem:2bound} to obtain the inequality below we have, provided $L_0 > 0$,
\begin{align} \label{eq:bound0}
\psi_0(x)  = \mathbb{P} \left( K(w) > L_0 , K(w') > L_0 | K(v_0) = x \right)  \leq C_q \exp \left( - \lambda_q (L_0 - q^m x  ) \right).
\end{align} 
Combining \eqref{eq:genrep} with \eqref{eq:erep}, and then using the bounds \eqref{eq:bound1} and \eqref{eq:bound0}, we have
\begin{align*}
\mathbb{P} \left( \bigcap_{i,j} \{ J(u_{i,j}) > t_{i,j} \} \right) \leq C_q e^{ - (\nu - 1 + \lambda_q) L_0 }  \mathbb{E} \left[\exp \left(  \lambda_q q^j K(v_0) + q^j \sum_{ i = 1}^{ \nu - 1 } K(v_i) \right) \right].
\end{align*}
Let $m\in\N$ be sufficiently large so that $ \lambda_q \nu q^m < 2 \nu q^m < 1/2$ holds. We have
\begin{align*}
 \mathbb{E} \left[\exp \left(  \lambda_q q^j K(v_0) + q^j \sum_{ i = 1}^{ \nu - 1 } K(v_i) \right) \right] \leq 
\mathbb{E} \left[\exp \left(\frac{1}{2 \nu} \sum_{ i = 0}^{ \nu - 1 } K(v_i) \right) \right]  \leq \mathbb{E} \left[ e^{ \frac{1}{2} K_\infty } \right] = C_q,
\end{align*} 
where we used the fact that $\exp\left(\frac{1}{n} \sum\limits_{i=0}^{n-1} a_i\right) \leq \frac{1}{n}\sum\limits_{i=0}^{n-1} \exp(a_i)$ for the second inequality.
In particular, we have
\begin{align*}
\mathbb{P} \left( \bigcap_{i,j} \{ J(u_{i,j}) > t_{i,j} \} \right) \leq C_q e^{ - (\nu - 1 + \lambda_q) L_0 }. 
\end{align*}
Since $L_0 = t_0 + \gamma n$, the claim follows.
\end{proof}

%%%%%%%%%%%%%%%%%%%%%%%%%%%%%%%%%
%%%%%%%%%%%%%%%%%%%%%%%%%%%%%%%%%
%%%%%%%%%%%%%%%%%%%%%%%%%%%%%%%%%
%%%%%%%%%%%%%%%%%%%%%%%%%%%%%%%%%
%%%%%%%%%%%%%%%%%%%%%%%%%%%%%%%%%
\section{Left tails for the geometric sum of Exponentials}\label{sec:smalls}
%%%%%%%%%%%%%%%%%%%%%%%%%%%%%%%%%
%%%%%%%%%%%%%%%%%%%%%%%%%%%%%%%%%
%%%%%%%%%%%%%%%%%%%%%%%%%%%%%%%%%
%%%%%%%%%%%%%%%%%%%%%%%%%%%%%%%%%
%%%%%%%%%%%%%%%%%%%%%%%%%%%%%%%%%

{
In this section we work towards proving Theorem \ref{thm:smalls}. Before presenting the proof of Theorem~\ref{thm:smalls}, we will give two preliminary lemmas.}

\begin{lemma}\label{lem:SimplexBound}  
For all $m\in\N$ and $s \geq 0$, we have 
\begin{equation} \label{eq:SimplexBound}
\frac{ s^m}{ m!} q^{ - m(m-1)/2} \exp \left( - \frac{ s q^{ - m} }{(q^{-1}-1)m} \right)\leq \mathbb{P} \left( K_{m-1} \leq s \right) \leq \frac{ s^m}{ m!} q^{ - m(m-1)/2} .
\end{equation}
\end{lemma}
\begin{proof} From the definition \eqref{eq:2:kndef} of $K_{m-1}$, we see that
\begin{equation*}
\mathbb{P}\left( K_{m-1} \leq s \right) = \int_{\mathbb{R}_+^{m}} 1_{\{u_0 + u_1+\cdots u_{m-1} \leq s\}} \prod_{i=0}^{m-1} q^{-i} \exp(-q^{-i}u_i) \mathrm{d} u_1  \mathrm{d} u_2 \cdots  \mathrm{d} u_m
\end{equation*} holds for all $s \geq 0$. The integral is taken over the $s$-scaled unit $m$-simplex of volume $s^m/m!$ (see \eqref{eq:s-scaled simplex}),
\begin{align*}
s \Delta^m = \left\{ (u_0,\ldots,u_{m-1})\,:\, u_i \geq 0,\, \sum_{ i = 0}^{m-1} u_i \leq s \right\}.
\end{align*}
In particular, we may write
\begin{align*}
\mathbb{P} \left( K_{m-1} \leq s \right) = \frac{ s^m}{ m!} q^{ - m(m-1)/2}  \mathbb{E} \left[ \exp \left( -  \sum_{ i = 0}^{m-1} q^{-i} \zeta_i \right) \right],
\end{align*}
where $(\zeta_0,\ldots,\zeta_{m-1})$ is a random vector uniformly distributed on $s \Delta^m$. 
The upper bound in \eqref{eq:SimplexBound} follows from the simple estimate 
$\mathbb{E} \left[ \exp \left( - s \sum_{ i = 0}^{m-1} q^{-i} \zeta_i \right) \right] \leq 1$. To prove the lower bound note that by Jensen's inequality, we have
\begin{align*}
\mathbb{E} \left[ \exp \left( -  \sum_{ i = 0}^{m-1} q^{-i} \zeta_i \right) \right] \geq \exp \left( - \mathbb{E} \left[  \sum_{ i = 0}^{m-1} q^{-i} \zeta_i \right] \right) .
\end{align*}
The lower bound in \eqref{eq:SimplexBound} now follows from noting that, since $\mathbb{E}[ \zeta_i ] = s/m$ for each $i\in\{0,\dots,m-1\}$, we have 
\begin{align*}
\mathbb{E} \left[  \sum_{ i = 0}^{m-1} q^{-i} \zeta_i \right]  = \frac{s}{m} \frac{q^{-m} - 1}{ q^{ - 1} - 1} \leq \frac{s}{m} \frac{q^{-m}}{ q^{-1} - 1}.
\end{align*}
\end{proof}
We emphasize that thanks to Lemma \ref{lem:SimplexBound}, it can be seen that whenever the quantity $\frac{ s q^{ - m}}{ (q^{-1} - 1) m}$ is small, 
the quantity $s^m q^{ - m(m-1)/2}/m!$ is a good estimate for $\mathbb{P} \left( K_{m-1} \leq s \right)$. Our proofs of both the upper and lower bounds in Theorem \ref{thm:smalls} will involve combining monotonicity arguments --namely that for $n \geq m$ $K_n$ stochastically dominates $K_m$ -- with taking an optimal choice of $m$.
For the latter, we have the following lemma, which identifies a critical choice of $m(s)$ so that $s^m q^{ - m(m-1)/2}/m!$ has the order $e^{ - F_q(s)}$. 
\begin{lemma} \label{lem:specialchoice}
For each $s\in(0,1/e^2]$, letting $m(s)$ be the smallest integer greater than $\kappa \left( \log \frac{1}{s} + \log \log \frac{1}{s} \right)$, we have 
\begin{align*}
\frac{1}{C_q} \exp \left( - F_q(s) \right) \leq s^{ m(s) } q^{ - m(s)(m(s)-1)/2} / m(s)! \leq C_q \exp \left( - F_q(s) \right),
\end{align*}
where $F_q(s)$ is as in Theorem \ref{thm:smalls}.
\end{lemma}

\begin{proof}
Note that by using the Stirling bounds $
\sqrt{ 2 \pi } m^{m+1/2} e^{ - m} \leq m! \leq e m^{m+1/2} e^{-m}
$, as well as the definition $q = e^{ - 1/\kappa}$, we have
\begin{align*} %\label{eq:fg}
	\frac{1}{C} \exp \left( f(s,m) \right) \leq s^{ m } q^{ - m(m-1)/2} / m!  \leq C \exp \left( f(s,m) \right),
\end{align*}
where for all $x,y>0$
\begin{align}\label{eq:f-defi}
	 f(x,y) := \frac{y^2}{2 \kappa} - \left( \log\frac{1}{x} - 1 + \frac{1}{ 2 \kappa} \right) y - (y + 1/2) \log y.
\end{align}
%Consider finding the value $x(s)$ minimising $f(s.x)$. Indeed, differentiating the function $f(s,x)$ with respect to $x$ we have 
%\dom{For the final version, I would shorten the calculations a little bit and drop the reasoning for the choice of $m(s)$ (currently commented out)}
%\begin{align*}
%\frac{ \partial f}{ \partial x} (s,x) = \frac{ x }{ \kappa } - \left( \log\frac{1}{s} - 1 + \frac{1}{ 2 \kappa} \right) - \log x - 1 - 1/2x.
%\end{align*}
%A computation that tells us that the solution $x(s)$ to $\frac{ \partial f}{ \partial x} \left( x(s), s\right) = 0$ takes the form 
%\begin{align*}
%x(s) = \kappa \left( \log \frac 1 s + \log \log \frac 1 s \right) + r(q,s).
%\end{align*} 
%where for each $q$, $r(q,s)$ is bounded in $s \geq 1/e^2$. With this guiding calculation in mind, 
%Let $m(s)$ be the smallest integer greater than $\kappa \left( \log \frac 1 s + \log \log \frac 1 s \right)$, so that $m(s)$ may be written $m(s) = \kappa \left( \log \frac 1 s + \log \log \frac 1 s  + \delta(s) \right)$ where $\kappa \delta(s) \in [0,1)$. 
Using the shorthand $S := \log \frac{1}{s}$, by setting $m(s) := \kappa \left( S + \log S + \delta(s) \right)$ (where $\delta(s) \in (0,1/\kappa]$ is such that $m(s)$ is an integer) a calculation tells us that 
\begin{align*} %\label{eq:mouse}
f(s,m(s)) = &- \frac{ \kappa }{2} \left( S + \log S + \delta(s) \right) \left( S + \log S - \delta(s) - 2 + \frac{1}{\kappa } + 2 \log \kappa  + 2 \varepsilon(s) \right) \notag \\
&- \frac{1}{2} \log S - \frac{1}{2} \log \kappa - \frac{1}{2} \varepsilon(s),
\end{align*}
where $\varepsilon(s) := \log m(s) - \log(\kappa S)= \log \left( 1 + \frac{ \log S + \delta(s) }{ S } \right)$. 
Using the fact that $\log(1+x) - x = \mathcal{O}(x^2)$ as $x \rightarrow 0$, we obtain 
\begin{align*} %\label{eq:elephant}
\left( S + \log S + \delta(s) \right) \varepsilon(s) = \log S + r(q,s),
\end{align*}
where for each $q$, $r(q,s)$ is bounded in $s \leq 1/e^2$. In particular, by the last two displays, we have
\begin{align*}
f(s,m(s))  = &   - \left( 1/2 + \kappa \right) \log S  \\ &- \frac{\kappa}{2} \left(S + \log S + \delta(s) \right) \left( S + \log S - \delta(s) - 2 + \frac{1}{ \kappa } + 2 \log \kappa \right)S + r'(q,s),
\end{align*} 
where $r'(q,s)$ is uniformly  bounded in $s \leq 1/e^2$. Using the identity $(x + a)(x + b) = (x + \frac{a+b}{2} )^2  - \left( \frac{a-b}{2} \right)^2$, we obtain
\begin{align*}
f(s,m(s)) =  - \left( \frac{1}{2} + \kappa \right) \log S - \frac{\kappa}{2} \left( S + \log S + \frac{1}{ 2 \kappa } + \log \kappa - 1 \right)^2  + r''(q,s)
\end{align*}
for a $r''(q,s)$ uniformly bounded in $s \leq 1/e^2$.
This completes the proof.
\end{proof}

%We now use the special choice $m(s) = \kappa \left( \log \frac{1}{s} + \log \log \frac{1}{s} \right)$ featuring in Lemma \ref{lem:specialchoice} to prove both the upper and lower bounds in Theorem \ref{thm:smalls}. We begin with the easier upper bound.

\begin{proof}[Proof of Theorem \ref{thm:smalls}]
Let $n\in\N$ be such that $n \geq \kappa \left( \log\frac{1}{s} + \log \log \frac{1}{s} \right)$. Then by construction, $n$ is at least $m(s)$ for all $n$ sufficiently large. Hence, by stochastic domination, we have 
\begin{align*}
\mathbb{P} \left( K_{n-1} \leq s \right) \leq \mathbb{P} \left( K_{m(s) - 1} \leq s \right).
\end{align*}
It then follows from the upper bound in Lemma \ref{lem:SimplexBound} and the upper bound in Lemma \ref{lem:specialchoice} that for every $s \leq 1/e^2$, 
\begin{align*}
\mathbb{P} \left( K_{n-1} \leq s \right) \leq \mathbb{P} \left( K_{m(s) - 1} \leq s \right) \leq C_q e^{ - F_q(s) },
\end{align*}
completing the proof of the upper bound in \eqref{eq:smalls}.\\

We now turn to proving the more difficult lower bound in \eqref{eq:smalls}.
Since $K_\infty$ stochastically dominates $K_{n-1}$ for every $n\in\N$, it is sufficient to prove the lower bound for $ n = \infty$. To this end, with $m(s)$ as in Lemma \ref{lem:specialchoice}, iterating~\eqref{eq:2:rde} $m(s)$ times yields
\begin{align*}
K_\infty \stackrel{(d)}{=} K_{m(s) - 1} + q^{m(s)} K_\infty, \qquad \mbox{$K_{m(s)-1}$ independent from $K_{\infty}$.}
\end{align*}
Our strategy is as follows. For a carefully chosen $\varepsilon(s)>0$, we use the bound
\begin{align} \label{eq:t4}
\mathbb{P} \left( K_\infty \leq s \right) \geq \mathbb{P} \left( K_{m(s)-1} \leq  \left(1 -  \varepsilon(s) \right)s \right) \mathbb{P} \left( q^{m(s)} K_\infty \leq \varepsilon(s) s \right).
\end{align}
It transpires that the best choice of $\varepsilon(s)$ to be taken is so that $\varepsilon(s) s q^{ - m(s)}$ has unit order. Indeed, with $S = \log \frac{1}{s}$ as above, set 
\begin{align*}
\varepsilon(s) := 1/S.
\end{align*} 
Using again $m(s) = \kappa \left( \log \frac{1}{s} + \log \log \frac{1}{s} + \delta(s) \right)$, a calculation tells us that
\begin{align*}
\varepsilon(s) s q^{ - m(s)} = e^{ \delta(s)  } \geq 1,
\end{align*} 
for every $s$, so that in particular
\begin{align*}
 \mathbb{P} \left( q^{m(s)} K_\infty \leq \varepsilon(s) s \right) = \mathbb{P} \left( K_\infty \leq \varepsilon(s) s q^{ - m(s) } \right) 
 \geq \mathbb{P} \left( K_\infty \leq 1 \right) \geq C_q.
\end{align*} 
Moreover, by \eqref{eq:t4} with $\varepsilon(s) = 1/S$, we have 
\begin{align} \label{eq:combo}
\mathbb{P} \left( K_\infty \leq s \right) \geq C_q \mathbb{P} \left( K_{m(s)-1} \leq  \left(1 -  \varepsilon(s) \right)s \right) .
\end{align}
%It remains to obtain a lower bound on $\mathbb{P} \left( K_{m(s)-1} <  \left(1 -  \epsilon(s) \right)s \right) $. In fact, what we will find is that $\epsilon(s)$ has been chosen sufficiently small so that $\mathbb{P} \left( K_{m(s)-1} <  \left(1 -  \epsilon(s) \right)s \right) $ has the same order as $\mathbb{P} \left( K_{m(s)-1} <  s \right) $. 
By Lemma \ref{lem:SimplexBound}, we can write
\begin{align} \label{eq:lowersimplex}
\mathbb{P} \left( K_{m-1} \leq s \right) \geq C_q \exp \left( f(s,m) - g(s,m ) \right), 
\end{align} 
where $f(s,m)$ is given as in \eqref{eq:f-defi}, and 
\begin{align*}
g(s,m) := \frac{ e^{ \kappa^{-1} m } s }{ (q^{-1}- 1) m }  .
\end{align*} 
Set $w(s) := (1 - \varepsilon(s)) s$ and let $m(s) = \kappa \left( S + \log S + \delta(s) \right)$ be defined as above. A calculation yields
\begin{align*} %\label{eq:gbound}
g( w(s), m(s) ) \leq \frac{ e^{ \delta(s)}}{ q^{-1} - 1 } \leq C_q  .
\end{align*} 
We now turn to computing $f(w(s),m(s))$. Again, a calculation similar to the one in the proof of Lemma \ref{lem:specialchoice} tells us that 
\begin{align*} %\label{eq:fbound}
f( w(s) , m (s)) = F_q(s) + r(q,s)
\end{align*}
where for each $q$, $r(q,s)$ is bounded uniformly in $s \leq 1/e^2$. In particular, we see that the difference between $f(w(s),m(s))$ and $f(s,m(s))$ is bounded. Using \eqref{eq:lowersimplex}, we have that
\begin{align*}
\mathbb{P} \left( K_{m(s) - 1} \leq w(s) \right) \geq C_q \exp \left( - F_q(s) \right) ,
\end{align*} 
and by \eqref{eq:combo} 
\begin{align*}
\mathbb{P} \left( K_\infty \leq s \right) \geq C_q \exp \left( - F_q(s) \right)  .
\end{align*}
This completes the proof of the lower bound in \eqref{eq:smalls}.
\end{proof}

%%%%%%%%%%%%%%%%%%%%%%%%%%%%%%
%%%%%%%%%%%%%%%%%%%%%%%%%%%%%%
%%%%%%%%%%%%%%%%%%%%%%%%%%%%%%
%%%%%%%%%%%%%%%%%%%%%%%%%%%%%%
\section{The leftmost particles in the rescaled expanding branching random walk } \label{sec:smallest}
\subsection{Three preliminary estimates}
%%%%%%%%%%%%%%%%%%%%%%%%%%%%%%
%%%%%%%%%%%%%%%%%%%%%%%%%%%%%%

We now give three simple estimates on the random variables $(K(v))_{v \in \T}$, which will be used in the proof of Theorem~\ref{thm:smallestnew}. Combined with the estimates on the lower tails of $K(v)$ from Theorem \ref{thm:smalls},  this will allow us to determine a sharp concentration of the size of the smallest fragment {in Theorem \ref{thm:smallest}.}
% Recall that $K_{n}^{\min}$ denotes the minimal value of $K_{n}^v$ among all $v \in V_n$ and set for all $n\geq 1$
%\begin{equation}\label{def:GenericKn}
%K_n := \sum_{i = 0}^{n} q^{i} W_i
%\end{equation} for independent standard exponential random variables $(W_i)$. Moreover, for all sites $v,w \in V_n$, we denote by $v \wedge w$ the least common ancestor of $v,w$.  
The following lemma gives an upper bound on the joint tails of $K(v)$ and $K(w)$, which will be useful when $w$ is close to $v$.

\begin{lemma}\label{lem:Decoupling1Segment} Let $v,w \in \T_n$ and suppose that $|v \wedge w| = n-m-1$ for some $m\geq 0$. Then, for all $x\geq 0$, we have that
\begin{equation*}
\P \left( K(v) \leq x, K(w) \leq x\right) \leq \P\left( K_n \leq x \right) \P\left( K_m \leq x \right)  .
\end{equation*}
\end{lemma}
\begin{proof}
Recall the expanding branching random walk $S = (S(x))_{x \in \T}$ defined in Section~\ref{sec:results}. We have
\begin{align*}
	\P(S(v)\leq t, \: S(w) \leq t) & = \P(S(v) \leq t, \: S(v \wedge w ) + S(w)-S(v \wedge w) \leq t ) \\
		& \leq  \P(S(v) \leq t, \:  S(w)-S(v \wedge w) \leq t ) =  \P(S(v) \leq t) \P(S(w)-S(v \wedge w) \leq t ).
\end{align*}
Now use the fact that $\{q^n S(v), q^nS(w)\} = \{K(v), K(w)\}$, where we have $q^nS(v) \stackrel{d}{=} K_n$ as well as that $S(w)-S(v \wedge w) \stackrel d= q^{-(n-m)}S_m\stackrel d= q^{-n}K_m$, and conclude by substituting $t = q^{-n}x$.

%
% Recall that $f_n$ from \eqref{def:Densityfn} denotes the density of $K_n$. By the construction of $K_{n}^v$ and $K_{n}^w$, we have that \begin{align}
%\P \left( K_{n}^v \leq x, K_{n}^w \leq x\right) &= \int_{0}^{x} \P \left( K_{n}^v \leq x-s q^{m} \mid K_{n-m-1}^{v\cap w} = s\right)^2 f_n(s) \dif s \notag\\ 
%&\leq \P \left( K_{m} \leq x\right) \int_{0}^{x} \P \left( K_{n}^v \leq x-s \mid K_{n-m-1}^{v\cap w} = s\right) f_n(s) \dif s  \\
%&=  \P\left( K_{m} \leq x \right) \P\left( K^v_{n} \leq x \right)  \notag
%\end{align} holds, using the self-similarity of $K_{n}^v$ for the inequality in the second line.
%
%
%Recall that $= \sum_{i=0}^{n-1} W_i^{w}$ for standard exponential random variables $(W_i^{w})$ and that $W_i^{w}$ is independent of $K_{n}^v$ for $i\geq n-m$. For $i<n-m$, replace the random variables $W_i^{w}$ by $0$ in the construction of $K_{n}^w$ to get 
%\begin{equation}
%\P \left( K_{n}^w \leq x | K_{n}^v \leq x \right) \leq   \P\left( K_{m} \leq x \right)
%\end{equation} using the FKG-property for $ K_{n}^w$.
\end{proof}

%The next lemma gives a bound on the tails of $K(v)$ and $K(w)$, which we will apply for distantly related sites.

\begin{lemma}\label{lem:Decoupling2Ancestor} Let $v,w \in \T_n$ and suppose that $|v \wedge w| = n-m-1$ for some $m\geq 0$. Then, for all $s,x\geq 0$, we have
\begin{equation*}%\label{eq:EstimateDecoupling2}
\P \left( K(v) \leq s, K(w) \leq s\right) \leq \P \left( K(v) \leq s\right)\left(\P\left( K(v) \leq s + q^{m+1} x \right) +  \P \left( K(v \wedge w) > x \right) \right).
\end{equation*}
\end{lemma}
\begin{proof} Lemma \ref{lem:Decoupling1Segment} gives
$\P \left( K(v) \leq s, K(w) \leq s\right) \leq \P \left( K(v) \leq s\right)
\P\left( K_{m} \leq s \right)$.
For a pair of independent random variables $(K_{m},K_{n-m-1})$ as defined in \eqref{eq:2:kndef}, set $\widetilde{K}_{n}:= K_{m} + q^{m+1} K_{n-m-1}$. Note that we have
\begin{equation*}
\P\left( K_{m} \leq s \right) \leq  \P\left( K_{m} \leq s , K_{n-m-1} \leq x\right) + \P\left(K_{n-m-1} > x\right) \leq  \P\left( \widetilde{K}_{n} \leq s + q^{m+1} x \right)  + \P\left(K_{n-m-1} > x\right)
\end{equation*} for all $x,s \geq 0$. Since $\widetilde{K}_{n}$ has the same law as $K(v)$, we conclude the proof. 
%
%
%Using the law of total probability and the construction of $K_{n}^v$ and $K_{n}^w$, we have that
%\begin{equation}
%\P \left( K_{n}^v \leq s, K_{n}^w \leq s\right) = P\left( K_{n-m-1}^{v \cap w} > x \right) + \sum_{i=m-n}^{\lfloor\log_q x \rfloor} \P \left( K_{n}^v \leq s \mid K_{n-m-1}^{v \cap w} = q^i \right)^2 \P\left( K_{n-m-1}^{v \cap w} = q^{i} \right)  
%\end{equation} holds for all $s,x\geq 0$. Observe that 
%\begin{equation}
%\P \left( K_{n}^v \leq s\right) = \P \left( K_{n}^v \leq s | K_{n-m-1}^{v \cap w} > x \right)  \leq 
%\end{equation}
%using the FKG-property and a coupling for the second step. This yields
%\begin{equation}
%\P \left( K_{n}^v \leq s \mid K_{n-m-1}^{v \cap w} = q^i \right) = \P \left( K_{n} \leq s + q^i \right)
%\end{equation}
%for all $i \in \{m-n,\lfloor\log_q x \rfloor \}$, we conclude \eqref{eq:EstimateDecoupling2}.
\end{proof} The following lemma provides an estimate for the probability that $K(v)$ is contained in a small interval.

\begin{lemma}\label{lem:Decoupling3Continuity} For all $s,z \geq 0$ and $n \in \N$, we have that
\begin{equation*}
\P \left( K_n \in [s,s+z] \right) \leq z \P \left( K_{n-1}\leq  q^{-1}(s+z) \right).
\end{equation*}
\end{lemma}
\begin{proof} 
{
Consider the following general fact. If $A$ is any non-negative random variable and $W$ is an independent standard exponential random variable, then
\begin{align*}
\mathbb{P} ( A+ W \in [s, s+z] ) %\leq  \mathbb{P} ( W \in [s-A, s+z-A] | A \leq s+z )  \mathbb{P} ( A \leq s+z)  
\leq z \mathbb{P} ( A \leq s+z).
\end{align*}
The result in question follows from this general fact by noting that $K_n \stackrel{d}{=}  q K_{n-1} + W$ where $W$ is an independent standard exponential. 
}
\end{proof}

\subsection{Second moment method}

In order to prove Theorem \ref{thm:smallestnew}, we will apply the second moment method with respect to the following sum of indicator random variables 
\begin{equation*}
	M_n(s) =   \sum_{v \in \T_n} I_{ \{K(v)\leq s\}}.
\end{equation*} 
We start with the following bound on the expectation of $M_n(s)$.
{Define $z_n:=z_n(\kappa,\gamma)$ as the unique solution to the equation
	\begin{equation}\label{eq:znEquation}
		z_n + \log z_n + \frac{1}{2\kappa} +\log\kappa- 1 = \sqrt{\frac{2\gamma}{\kappa}n}\,.
	\end{equation} 
It is easily verified that
	\begin{equation}\label{eq:znasymptotics}
		z_n = \sqrt{\frac{2\gamma}{\kappa}n} - \frac{1}{2}\log n - \frac{1}{2\kappa}- \frac{1}{2}\log\kappa +1 - \frac{1}{2}\log(2\gamma) + O \left( \frac{ \log n }{ \sqrt{n}} \right)  .
	\end{equation} 
}

\begin{lemma}\label{lem:SecondMomentExpectation} {With $z_n$ as in \eqref{eq:znEquation}, for $n \in \mathbb{N}$ define the quantities}
\begin{equation}\label{def:s+-}
s_n^{-}:= \exp\left( -z_n - z_n^{-1}\log^2 z_n \right) 
\qquad\text{and}\qquad
s^{+}_n:= \exp\left(-  z_n + z_n^{-1}\log^2 z_n \right)
\end{equation} 
Then
\begin{equation*}
 \mathbb{E}\left[ M_n(s_n^{-}) \right] = k^n \P\left(K_n \leq s_n^{-} \right) \leq \frac{1}{n^2}
\qquad\text{and}\qquad
\mathbb{E}\left[ M_n(s_n^{+}) \right] = k^n \P\left(K_n \leq s_n^{+} \right)  \geq n^2
\end{equation*} for all $n\in\N$ sufficiently large.
\end{lemma}

\begin{proof} By the definition of $z_n$, we have that $F_q(s)$ given in \eqref{def:functionFqs} satisfies
\begin{equation*}
F_q(\exp(-z_n+y_n))= \left( \frac{1}{2}+\kappa \right)\log z_n + \frac{\kappa}{2}\left( \sqrt{\frac{2\gamma}{\kappa}n}+y_n \right)^{2} + O(y_n)
\end{equation*} for all $(y_n)_{n \in \N}$ with $\lim_{n \rightarrow \infty} y_n = 0$.
Hence, using Theorem \ref{thm:smalls}, we see that
\begin{equation*}
\log\left(k^n\P\left(\log K_n \leq - z_n - z_n^{-1}\log^2 z_n \right) \right) \leq - \frac{1}{2}\log n - \frac{\kappa}{2}\log^2 z_n \leq -2\log n
\end{equation*} for all $n$ large enough. Similarly, we apply Theorem \ref{thm:smalls} to obtain
\begin{equation*}
\log\left(k^n\P\left(\log K_n \geq - z_n + z_n^{-1}\log^2 z_n \right) \right) \geq - \frac{1}{2}\log n + \frac{\kappa}{2}\log^2 z_n \geq 2\log n
\end{equation*} for all $n$ large enough, which concludes the proof.
\end{proof}

We now have all tools to prove Theorem \ref{thm:smallestnew}.

\begin{proof}[Proof of Theorem \ref{thm:smallestnew}] 
{With a view to proving Theorem \ref{thm:smallestnew}, we begin by proving the slightly stronger statement that there exists almost surely an $n_0 \in \mathbb{N}$ such that for all $n \geq n_0$ we have 
	\begin{align}\label{eq:ConvergenceRateKmin}
		\log K^{\min}_n  \in \left[ -z_n - z_n^{-1}\log^2z_n, - z_n +z_n^{-1}\log^2z_n\right]
	\end{align}

To this end, we note} from Lemma \ref{lem:SecondMomentExpectation}, we see that for $s_n^-$ defined in \eqref{def:s+-}, we have
\begin{equation*}
\P\left( \exists  v \in \T_n \colon \log K(v)  \leq \log s_n^-\right) = \P\left(M_n(s_n^-)\geq 1\right)  \leq \mathbb{E}\left[ M_n(s_n^{-}) \right] \leq \frac{1}{n^2},
\end{equation*}
which gives the $\P$-almost sure lower bound on $\log K_{n}^{\min}$ in \eqref{eq:ConvergenceRateKmin}, i.e.\,
$\log K_{n}^{\min}
\geq \log s_n^-$ almost surely for $n\in\N$ large enough. For the corresponding upper bound, we will now estimate 
$\Var(M_n(s))$, where we set $s=s_n^{+}$. Partitioning according to the generation of the most recent common ancestor, we have that
\begin{align*}
\Var(M_n(s)) &= \sum_{v,w \in \T_n} \left( \P\left( K(v) \leq s, K(w) \leq s \right) - \P\left( K(v) \leq s \right)\P\left( K(w) \leq s \right) \right) \\
&\leq \E[M_n(s)] + \sum_{m=1}^{n-1} k^{n+m} \left( \P\left( K(v_m) \leq s, K(w_m) \leq s \right) - \P\left( K(v_m) \leq s \right)\P\left( K(w_m) \leq s \right) \right),
\end{align*} where $v_m,w_m \in \T_n$ are chosen for all $m\geq 1$ such that the equality $|v_m \wedge w_m| = n-m$ holds. Splitting the last sum at $n/2$, we have, using Lemma \ref{lem:Decoupling1Segment}, that 
\begin{align*}
\sum_{m=1}^{n/2-1} k^{n+m} \P\left( K(v_m) \leq s, K(w_m) \leq s \right) & \leq 
\sum_{m=1}^{n/2-1} k^{n} \P\left( K_n \leq s \right) k^m \P\left( K_m \leq s \right) \\&= \E[M_n(s)] \sum_{m=1}^{n/2-1}k^m \P\left( K_m \leq s \right) .
\end{align*} Recall that $z_n$ is of order $\sqrt{n}$. Hence, since $s^{-}_{m} \geq s^{+}_{n}$ for all $m\leq n/2$ when $n$ is large enough, we can use Lemma \ref{lem:SecondMomentExpectation} to see that $\sum_{m=0}^{n/2}k^m \P\left( K_m \leq s \right)$ is bounded from above uniformly in $n$. For the remaining terms, we apply Lemma \ref{lem:Decoupling2Ancestor} to get that
\begin{align*}
k^{n+m}\left(\P\left( K(v_m) \leq s, K(w_m) \leq s \right) -  \P\left( K(v) \leq s \right)^2 \right) \leq \E[M_n(s)]  k^{m}\left(\P\left( K_n \in [s, s+q^mx]\right)  + \P\left( K_\infty > x \right) \right)
\end{align*} holds for all $x \geq 0$. Choose $x=n^2$ and use Lemma \ref{lem:Decoupling3Continuity}  as well as Theorem \ref{thm:smalls} to obtain
\begin{align*}
 \sum_{m=n/2}^{n-1} k^{m}\left(\P\left( K_n \in [s, s+q^mn^2] \right)\right) &\leq 
 n^2 q^{n/2}\sum_{m=n/2}^{n-1} k^{m} \P\left( K_{n-1} \leq q^{-1}2s \right) \leq 1
\end{align*} for all $n$ large enough.
Note that $nk^m\P(K_{\infty}>n^2)\leq 1$ holds for all $m\leq n$ with $n$ sufficiently large by Lemma \ref{lem:tails}. Hence, combining the previous observations, we obtain that
\begin{align*}
\sum_{m=n/2}^{n-1} k^{n+m} \left(\P\left( K(v_m)\leq s, K(w_m) \leq s \right) -  \P\left( K_n \leq s \right)^2 \right)  \leq 2 \E[M_n(s)]
\end{align*} holds for all $x \geq 0$ and all $n$ large enough.
Thus, we conclude that the variance $\Var(M_n(s_n^+))$ is of order at most $\E[M_n(s_n^+)]$. Using the Paley–Zygmund inequality, we have, writing again $s = s_n^+$, 
\begin{equation*}
\P\left( \exists  v \in \T_n \colon \log K(v)\leq \log s \right) =  \P\left( M_n(s) >0 \right) \geq  \frac{\E[M_n(s)^2]}{\E[M_n(s)]^2} = 1 - 
 \frac{\Var(M_n(s))}{\E[M_n(s)]^2} \geq 1- \frac{c}{n^2}
\end{equation*} for all $n$ large enough, where we used Lemma \ref{lem:SecondMomentExpectation} for the last inequality. Applying the Borel-Cantelli lemma, this gives us the upper bound on $\log K_{n}^{\min}$ in \eqref{eq:ConvergenceRateKmin}, i.e.\, $\log K_{n}^{\min}
\leq \log s_n^+$ almost surely for n large enough.

{
We now obtain the statement in Theorem \ref{thm:smallest} from the stronger statement given in \eqref{eq:ConvergenceRateKmin}. First we note that with $w_n$ as in the statement of Theorem \ref{thm:smallest}, by \eqref{eq:znasymptotics} we have $w_n = z_n  + O( \frac{ \log n }{ \sqrt{n}}  )$. In particular, since $\frac{ \log^2 z_n }{ z_n} = O \left( \frac{ \log^2 n }{ \sqrt{n}} \right) = o( n^{-1/3} )$, we obtain that there exists almost surely an $n_0$ in $\mathbb{N}$ such that for all $n \geq n_0$ we have 
\begin{align*}
\log K_n^{\min} \in \left[ - w_n - \frac{1}{n^{1/3} },  - w_n + \frac{1}{ n^{1/3}}  \right]
\end{align*}
which is precisely the statement of Theorem \ref{thm:smallest}.
}
\end{proof}

\subsection{Proof of Theorem \ref{thm:smallest}}

We are now ready to use Theorem \ref{thm:smallestnew} to give a proof of Theorem \ref{thm:smallest}.

\begin{proof}[Proof of Theorem \ref{thm:smallest} using Theorem \ref{thm:smallestnew}]

{
Let $k^{ - M_t}$ denote the size of the largest fragment in the process at time $t$. Let
\begin{align*}
T_n := \sup \{ t \geq 0 : M_t = n \}
\end{align*}
denote the last time at which there was a fragment of size $k^{ - n}$ in the process. In particular, based on our discussion in Section \ref{sec:time}, $T_n := \min_{v \in \mathbb{T}_n} S(v)$. Since $S_n^{\min} = q^{ - n} K_n^{ \min}$, by Theorem \ref{thm:smallestnew} there exists almost surely some $n_0$ in $\mathbb{N}$ such that for all $n \geq n_0$ 
\begin{equation} \label{eq:oksendal}
\exp \left\{ \frac{n}{\kappa} - w_n - \frac{1}{n^{1/3}} \right\} \leq T_n \leq \exp \left\{ \frac{n}{\kappa} - w_n + \frac{1}{n^{1/3}} \right\} ,
\end{equation} 
where we made use of the fact that $q^{- n } = e^{ n/\kappa}$. Set $\tilde{c} := \frac{1}{ 2\kappa} + \frac{1}{2} \log \kappa - 1 + \frac{1}{2} \log (2\gamma)$ and for $\sigma \in \{-1,+1\}$ define 
\begin{equation} \label{eq:pdef}
 p_\sigma (x) := \exp \left\{ \frac{x}{\kappa} - \sqrt{\frac{2\gamma}{\kappa}x} +  \tilde{c} + \sigma x^{-1/3}\right\} . 
\end{equation}
Using the definition of $w_n$ given in the statement of Theorem \ref{thm:smallestnew}, \eqref{eq:oksendal} reads as saying
\begin{align*}
p_{-1}(n) \leq T_n \leq p_{+1}(n)
\end{align*}
for all $n \geq n_0$. In particular, we are in the setting of Lemma \ref{lem:convert}, so that there exists almost surely a $t_0$ in $\mathbb{R}_+$ such that for all $t \geq t_0$ 
\begin{equation} \label{eq:plusminus}
M_t \in \left\{ \lceil p_{+1}^{-1}(t) \rceil, \lceil p_{-1}^{ - 1}(t) \rceil \right\}.
\end{equation}
Setting $c := \tilde{c} + \frac{1}{2} \log \kappa - \gamma$ (so that it agrees with the constant $c$ given in Theorem \ref{thm:smallest}), a brief calculation inverting \eqref{eq:pdef} verifies that for $\sigma \in \{-1,+1\}$ we have
\begin{equation*}
p_\sigma^{-1}(t) := \kappa \left( \log t + \sqrt{ 2 \gamma \log t } - \frac{1}{ 2} \log \log t - c - \frac{\sigma}{ (\kappa \log t )^{1/3} } + o \left( \frac{1}{ \log^{1/3} t } \right) \right).
\end{equation*}
In particular, with $h(t)$ and $\mu_2 := 2 \kappa^{2/3}$  as in the statement of Theorem \ref{thm:smallest}, for all $t$ sufficiently large we have 
\begin{align*}
h(t) - \mu_2 \frac{ 1}{ \log^{1/3} t } \leq p_{+1}^{ - 1}(t) \leq p_{-1}^{-1}(t) \leq h(t) + \mu_2 \frac{1}{ \log^{1/3}t }.
\end{align*}
Moreover, for all $t$ sufficiently large we have $\{ \lceil p_{+1}^{-1}(t) \rceil, \lceil p_{-1}^{ - 1}(t) \rceil \} \subseteq \{ \ \lceil h(t) - \mu_2 \frac{ 1}{ \log^{1/3} t } \rceil , \lceil h(t) +  \mu_2 \frac{ 1}{ \log^{1/3} t } \rceil \}$. The statement of Theorem \ref{thm:smallest} now follows from \eqref{eq:plusminus}.  
}
 \end{proof}

%%%%%%%%%%%%%%%%%%%%%%%%%%%%%%%%%
%%%%%%%%%%%%%%%%%%%%%%%%%%%%%%%%%
%%%%%%%%%%%%%%%%%%%%%%%%%%%%%%%%%
%%%%%%%%%%%%%%%%%%%%%%%%%%%%%%%%%
%%%%%%%%%%%%%%%%%%%%%%%%%%%%%%%%%
%%%%%%%%%%%%%%%%%%%%%%%%%%%%%%%%%
\section*{Appendix}
In this section we provide a proof of Lemma \ref{lem:fkg}.
\begin{proof}[Proof of Lemma \ref{lem:fkg}]
%	Following the proof of \cite[Lemma 5.2]{gantert2018large}, 
We begin with the general observation stated in \eqref{replacebyone}. 
%Let $(U_i)_{i\in \N}$ and $(V_i)_{i \in \N}$ be independent sequences of (not necessarily independent) random variables. Furthermore, assume that the random variables $(V_i)_{i \in \N}$ all have the same law.
%	Then we have for all $k\in \N$, and any choice of $x_1, x_2, \ldots x_k \in \R$, 
%	\begin{equation}\label{eq:4:observation}
%		\P \left( \bigcap_{i=1}^k \left\{ U_i+V_i  > x_i  \right\}\right)\leq \P \left( \bigcap_{i=1}^k \left\{ U_i+V_1  > x_i  \right\}\right).
%	\end{equation}
%	Indeed, let $i_0 \in \{ 1, \ldots , k\}$ denote the smallest (random) index for which $U_{i_0} - x_{i_0} = \min_{1\leq i \leq k} \{ U_i-x_i \}$, and we can write
%	\begin{align*}
%		\P \left( \bigcap_{i=1}^k \left\{ U_i+V_i  > x_i  \right\}\right) & = \P \left( \min_{1 \leq i \leq k} \left\{ U_i-x_i+V_i   \right\} >0\right) \leq  \P \left(  U_{i_0}-x_{i_0}+V_{i_0}   >0\right) \\
%			& = \P \left(  U_{i_0}-x_{i_0}+V_{1}   >0\right) =  \P \left( \min_{1 \leq i \leq k} \left\{ U_i-x_i+V_1   \right\} >0\right) \\
%			& =\P \left( \bigcap_{i=1}^k \left\{ U_i+V_1  > x_i  \right\}\right).
%	\end{align*}
%A proof can be found in \cite[Lemma 5.2]{gantert2018large}.	
Using~\eqref{replacebyone}, we will now prove the following inequality for $(S(v))_{v \in \T}$ (defined in \eqref{recforS}):
	\begin{align} \label{eq:rednew}
		\mathbb{P} \left( \bigcap_{ i = 0}^{ \ell - 1 } \bigcap_{ v \in \T_{n+i}}\{ S(v) > t_{v} \} \right) \geq \prod_{ i = 0}^{ \ell - 1} \prod_{ v \in \T_{n+i}}  \mathbb{P} \left( S(v) > t_{v} \right)
	\end{align}
	for any choice of real numbers $t_v$, where $v\in \bigcup_{j=n}^{n-\ell-1} \T_j$. First invoke the branching property to write
	\begin{equation}\label{eq:4:neq1}
		\mathbb{P} \left( \bigcap_{ i = 0}^{ \ell - 1 } \bigcap_{ v \in \T_{n+i}}\{ S(v) > t_{v} \} \right) = \prod_{w \in \T_1} \mathbb{P} \left( \bigcap_{ i = 0}^{ \ell - 1 } \bigcap_{ v \in \T_{n+i}, \: v\geq w}\{ S(v) > t_{v} \} \right).
	\end{equation}
	For $w \in \T_1$ and $v \in \T_{n+i}$ with  $v\geq w$, we have $S(v) = S(w) + (S(v)-S(w))$, where $S(w)$ is independent from $(S(v)-S(w))_{ v \in \bigcup_{i=0}^{\ell-1}\T_{n+i}, \: v\geq w}$. Applying~\eqref{replacebyone}, we have that
	\begin{align}
		\mathbb{P} \left( \bigcap_{ i = 0}^{ \ell - 1 } \bigcap_{ v \in \T_{n+i}, \: v\geq w}\{ S(v) > t_{v} \} \right) &= \mathbb{P} \left( \bigcap_{ i = 0}^{ \ell - 1 } \bigcap_{ v \in \T_{n+i}, \: v\geq w}\{ S(v)-S(w)+S(w) > t_{v} \} \right)\nonumber\\ 
			& \geq \mathbb{P} \left( \bigcap_{ i = 0}^{ \ell - 1 } \bigcap_{ v \in \T_{n+i}, \: v\geq w}\{ S(v)-S(w)+S_v(w) > t_{v} \} \right) ,\label{eq:4:goback}
	\end{align}
	where $(S_v(w))_{ v \in \bigcup_{i=0}^{\ell-1}\T_{n+i}, \: v\geq w}$ are i.i.d.\ copies of $S(w)$. Note that
	\begin{equation*}
		\mathbb{P} \left( \bigcap_{ i = 0}^{ \ell - 1 } \bigcap_{ v \in \T_{n+i}, \: v\geq w}\{ S(v)-S(w) > u_{v} \} \right) = \mathbb{P} \left( \bigcap_{ i = 0}^{ \ell - 1 } \bigcap_{ v \in \T_{n-1+i}, }\{ S(v)> u_{v} \} \right).
	\end{equation*}
	Now we can go back to~\eqref{eq:4:goback}, condition on the values of $S_v(w)$ and apply~\eqref{eq:4:neq1}. Repeating this procedure $n$ times yields 
	\begin{equation}\label{eq:4:neqn}
		\mathbb{P} \left( \bigcap_{ i = 0}^{ \ell - 1 } \bigcap_{ v \in \T_{n+i}}\{ S(v) > t_{v} \} \right) \geq \prod_{w \in \T_n} \mathbb{P} \left( \bigcap_{ i = 0}^{ \ell - 1 } \bigcap_{ v \in \T_{n+i}, \: v\geq w}\{ S(v) > t_{v} \} \right).
	\end{equation}
	To obtain~\eqref{eq:rednew}, it is therefore sufficient to prove that, for all $w \in \T_n$,
	\begin{equation}\label{eq:4:neq0}
		\mathbb{P} \left( \bigcap_{ i = 0}^{ \ell - 1 } \bigcap_{ v \in \T_{n+i}, \: v\geq w}\{ S(v) > t_{v} \} \right) \geq \prod_{ i = 0}^{ \ell - 1} \prod_{ v \in \T_{n+i}, \: v \geq w}  \mathbb{P} \left( S(v) > t_{v} \right).
	\end{equation}
	We can achieve this by using the same arguments as above since 
	\begin{align*}
		\mathbb{P} \left( \bigcap_{ i = 0}^{ \ell - 1 } \bigcap_{ v \in \T_{n+i}, \: v\geq w}\{ S(v) > t_{v} \} \right) &= \mathbb{P} \left( \bigcap_{ i = 0}^{ \ell - 1 } \bigcap_{ v \in \T_{n+i}, \: v\geq w}\{ S(v)-S(w)+S(w) > t_{v} \} \right)\nonumber\\ 
			& \geq \mathbb{P} \left( \bigcap_{ i = 0}^{ \ell - 1 } \bigcap_{ v \in \T_{n+i}, \: v\geq w}\{ S(v)-S(w)+S_v(w) > t_{v} \} \right) \nonumber \\
			& =  \P(S(w) > t_w)\mathbb{P} \left( \bigcap_{ i = 1}^{ \ell - 1 } \bigcap_{ v \in \T_{n+i}, \: v\geq w}\{ S(v)-S(w)+S_v(w) > t_{v} \} \right).
	\end{align*}
	Using the same inductive procedure as before gives~\eqref{eq:4:neq0}. In combination with~\eqref{eq:4:neqn}, this  yields~\eqref{eq:rednew}. The claim now follows by taking in \eqref{eq:rednew} $t_v \to -\infty$ if $u_{i,j}\neq v$ for all $u_{i,j}\in \textbf{u}$, and $t_v = t_{i,j}q^{|v|}-\gamma |v|$ if $v= u_{i,j}$.
\end{proof}

\subsection*{Acknowledgement}
SJ and JP are supported by the Austrian Science Fund (FWF) Project P32405 \textit{Asymptotic geometric analysis and applications} of which JP is principal investigator. DS thanks the Studienstiftung des deutschen Volkes and the TopMath program for financial support. The research of PD was supported by the Alexander von Humboldt Foundation.
We thank G\"unter Last for answering questions about point processes.
%%%%%%%%%%%%%%%%%%%%%%%%%%%%%%%%%
%%%%%%%%%%%%%%%%%%%%%%%%%%%%%%%%%
%%%%%%%%%%%%%%%%%%%%%%%%%%%%%%%%%
%%%%%%%%%%%%%%%%%%%%%%%%%%%%%%%%%
%%%%%%%%%%%%%%%%%%%%%%%%%%%%%%%%%
%%%%%%%%%%%%%%%%%%%%%%%%%%%%%%%%%

\bibliography{geometric_fragmentations}
\bibliographystyle{abbrv}

\par\bigskip
\begin{footnotesize}
  \begin{tabular}{l}%
    \textsc{Piotr Dyszewski, Nina Gantert and Dominik Schmid}\\
    {Fakultät für Mathematik, Technische Universität München}\\
    \textit{E-mail addresses}:
     \href{mailto:{piotr.dyszewski@tum.de}}{piotr.dyszewski@tum.de}, \href{mailto:{gantert@ma.tum.de}}{gantert@ma.tum.de}, \href{mailto:{dominik.schmid@tum.de}}{dominik.schmid@tum.de}
  \end{tabular}
\par \medskip
  \begin{tabular}{l}%
    \textsc{Samuel G.G.\ Johnston}\\
    {Department of Mathematical Sciences, University of Bath}\\
    \textit{E-mail address}:
      \href{mailto:{sgj22@bath.ac.uk}}{sgj22@bath.ac.uk}
  \end{tabular}
\par\medskip
  \begin{tabular}{l}%
    \textsc{Joscha Prochno}\\
    {Institute of Mathematics and Scientific Computing, University of Graz}\\
    \textit{E-mail address}:
 \href{mailto:{joscha.prochno@uni-graz.at}}{joscha.prochno@uni-graz.at}
  \end{tabular}
\end{footnotesize}

\end{document}